\documentclass[11pt]{amsart}
\usepackage{CJKutf8}
\usepackage{amsmath,amsthm,amssymb} 
\usepackage{hyperref}
\usepackage{geometry}
\usepackage{url}
\usepackage{bm}
\usepackage{esint}
\usepackage{afterpage}
\usepackage{centernot}
\usepackage[T1]{fontenc} 
\usepackage{textcomp} 
\usepackage{footnote}
\usepackage{color}
\usepackage{mathrsfs}
\usepackage{leftidx}
\usepackage{bm}
\usepackage{url}
\usepackage{comment}
\usepackage{gensymb}
\usepackage{tikz}
\usepackage{pgfplots}
\usepackage{smartdiagram}
\usetikzlibrary{calc}
\usepackage{graphics}
\usepackage{tikz-cd}
\usetikzlibrary{arrows} % required in the preamble

\DeclareMathOperator{\tr}{Tr}

\DeclareMathOperator{\ad}{ad}

\DeclareMathOperator{\inj}{inj}

\DeclareMathOperator{\Vol}{Vol}

\DeclareMathOperator{\Span}{Span}

\DeclareMathOperator{\SO}{SO}

\DeclareMathOperator{\GL}{GL}

\DeclareMathOperator{\T}{T}

\DeclareMathOperator{\Rm}{Rm}

\newcommand{\R}{\mathbb R}

\newcommand{\Z}{\mathbb Z}

\renewcommand{\T}{\mathbb T}

\newcommand{\cT}{\mathcal{T}}

\newcommand{\diff}{\text{\rm d}}

\newcommand{\del}{\partial}

\theoremstyle{plain}

\newtheorem{theorem}{Theorem}
\newtheorem{proposition}[theorem]{Proposition}
\newtheorem{lemma}[theorem]{Lemma}

\newtheorem{conjecture}[theorem]{Conjecture}
\newtheorem{question}[theorem]{Question}

\theoremstyle{definition}
\newtheorem{definition}[theorem]{Definition}
\newtheorem{remark}[theorem]{Remark}

\theoremstyle{plain}
\newtheorem*{solution*}{Solution}
\newtheorem*{theorem*}{Theorem}
\newtheorem*{proposition*}{Proposition}
\newtheorem*{lemma*}{Lemma}
\newtheorem*{corollary*}{Corollary}
\newtheorem*{conjecture*}{Conjecture}
\theoremstyle{definition}
\newtheorem*{definition*}{Definition}
\newtheorem*{remark*}{Remark}
\newtheorem*{remarks*}{Remarks}
\numberwithin{equation}{section}
\numberwithin{theorem}{section}

\author[J. Fine]{Joel Fine}
\address{D\'epartement de Math\'ematiques, Universit\'e Libre de Bruxelles, Brussels, Belgium;}
\email{joel.fine@ulb.be}
\author[W.-Y. He]{Weiyong He}
\address{Department of Mathematics, University of Oregon, Eugene, OR 97403, USA}
\email{whe@uoregon.edu}
\author[C.-J. Yao]{Chengjian Yao}
\address{Institute of Mathematical Sciences, ShanghaiTech University, Pudong New District,
	Shanghai, 201210, China.}
\email{yaochj@shanghaitech.edu.cn}

\begin{document}

	\title{Convergence of the hypersymplectic flow on $\T^4$ with $\T^3$-symmetry}
%	\author{Joel Fine, Weiyong He, Chengjian }
	\date{\today}

 \begin{abstract}
 	A hypersymplectic structure on a 4-manifold is a triple $\omega_1, \omega_2, \omega_3$ of 2-forms for which every non-trivial linear combination $a^1\omega_1 + a^2 \omega_2 + a^3 \omega_3$ is a symplectic form. Donaldson has conjectured that when the underlying manifold is compact, any such structure is isotopic in its cohomolgy class to a hyperk\"ahler triple. We prove this conjecture for a hypersymplectic structure on $\T^4$ which is invariant under the standard $\T^3$ action. The proof uses the hypersymplectic flow, a geometric flow which attempts to deform a given hypersymplectic structure to a hyperk\"ahler triple. We prove that on $\T^4$, when starting from a $\T^3$-invariant hypersymplectic structure, the flow exists for all time and converges modulo diffeomorphisms to the unique cohomologous hyperk\"ahler structure. %This generalises work of Huang--Wang--Yao. In the case of hypersymplectic flow on $\mathbb{T}^4$ of normal form, the flow is shown to converge modulo diffeomorphisms to the standard flat hypersymplectic structure determined by the initial data.
 \end{abstract}

	\maketitle	

\section{Introduction}\label{introduction}

\subsection{Overview}

A hypersymplectic structure $\underline\omega$ on a smooth $4$-manifold $X^4$ is a triple of closed 2-forms $\underline\omega=\left(\omega_1, \omega_2,\omega_3\right)$ for which $a^1\omega_1+a^2\omega_2+a^3\omega_3$ is symplectic for any $(a^1,a^2,a^3)\in \R^3\setminus\{ 0\}$. The simplest example is the triple of K\"ahler forms of a hyperk\"ahler metric. Donaldson has conjectured that up to isotopy, on a compact manifold this is the \emph{only} example \cite{Don}:

\begin{conjecture}[Donaldson]\label{Donaldson-conjecture}
Let $\underline{\omega} = (\omega_1, \omega_2, \omega_3)$ be a hypersymplectic structure on a compact 4-manifold $X$ with $\int \omega_i \wedge \omega_j = 2 \delta_{ij}$. Then $\underline{\omega}$ can be deformed through cohomologous hypersymplectic structures to the triple of K\"ahler forms coming from a hyperk\"ahler metric on $X$. 
\end{conjecture}

(Notice that given any hypersymplectic structure, we can act by a constant linear transformation on the $\omega_i$ to ensure that $\int \omega_i \wedge \omega_j = 2 \delta_{ij}$. The factor of 2 here is just so the hyperk\"ahler metric---if it exists!---has unit volume.)

As we will explain in \S\ref{hypersymplectic-flow-recap}, a hypersymplectic manifold automatically has $c_1=0$ (where $c_1$ is defined via any of the symplectic forms $a^i\omega_i$). The classification of symplectic 4-manifolds with $c_1=0$ is an important problem which, despite much progress, is still largely open. For example, the following question appears to be currently out of reach:

\begin{question}\label{torus-question}
Let $\omega$ be a symplectic form on $\T^4$ with $c_1=0$. Is $\omega$ the K\"ahler form of a flat metric? In other words, does there exist an $\omega$-compatible complex structure $J$ for which $(\omega,J)$ is flat?
\end{question}

In fact, it follows from Taubes's work on the Seiberg--Witten invariants of symplectic 4-manifolds, that a symplectic form on $\T^4$ automatically has $c_1=0$ \cite{Taubes1,Taubes2}. We also remark that it suffices to find an $\omega$-compatible integrable complex structure $J_0$ on $\T^4$. This is because the cohomology class $[\omega]$ then contains a K\"ahler form $\omega_0 \in [\omega]$ for which $(\omega_0,J_0)$ is flat; now, since the set of K\"ahler metrics in $[\omega]$ is connected, Moser's trick gives a diffeomorphism $\Phi$ with $\Phi^*\omega_0 = \omega$ and then $J=\Phi^*J_0$ makes $(\omega, J)$ flat.

If we suppose that the symplectic form in Question~\ref{torus-question} comes from a hypersymplectic structure on $\T^4$,  $\omega = a^i \omega_i$, then Donaldson's conjecture (together with Moser's trick) implies that $\omega$ is indeed the K\"ahler form of a flat metric  (since hyperk\"ahler metrics on $\T^4$ are necessarily flat).

We recommend the excellent survey article \cite{Tian-Jun-Li} for a thorough discussion of the general problem of classifying symplectic 4-manifolds with $c_1=0$. Donaldson's conjecture is a special case of this problem which one might realistically hope is more tractable. One reason for this hope is the \emph{hypersymplectic flow}, a geometric flow which attempts to directly carry out an isotopy from a given hypersymplectic triple to a hyperk\"ahler triple. This flow was introduced in \cite{FY}; we recall the definition in \S\ref{hypersymplectic-flow-recap} below. 

The main result of this article is that for a certain type of $\T^3$-invariant hypersymplectic structure on $\T^4$, the hypersymplectic flow proves Donaldson's conjecture: starting at such a $\T^3$-invariant hypersymplectic structure, the flow exists for all time and converges to a hyperk\"ahler triple in the limit modulo diffeomorphisms. In particular, for these symplectic forms we give a positive answer to Question~\ref{torus-question}. See Theorem~\ref{main-result} in \S\ref{introduction-main-result} below for a precise statement. The rest of the introduction sets the scene and includes a brief discussion of how Theorem~\ref{main-result} compares with other  results in the literature. 

\subsection{The hypersymplectic flow}\label{hypersymplectic-flow-recap}

We now recall the definition of the hypersymplectic flow. (The details can be found in \cite{FY}.) First we explain how a hypersymplectic structure $\underline{\omega}$ determines a Riemannian metric $g_{\underline{\omega}}$ on $X$. Write $V = \text{Span}\left\{\omega_1,\omega_2,\omega_3\right\} \subset \Lambda^2 T^*X$. The fact that $\underline{\omega}$ is hypersymplectic implies that $V$ is a rank 3 sub-bundle which is \emph{definite} under the wedge product. In other words, if we pick a nowhere-vanishing 4-form $\mu$, then the symmetric bilinear form on $V$ defined by 
\[
(\alpha, \beta) = \frac{\alpha \wedge \beta}{\mu}
\]
is either positive or negative definite. Requiring this to be \emph{positive} definite determines an orientation on $X$. It is then a standard fact that there is a unique  conformal structure on $X$ for which $V = \Lambda^+$ is the bundle of self-dual 2-forms. 

We remark in passing that this also shows why $c_1(a^i\omega_i)=0$ (for any $(a^i) \neq 0$): given a conformal 4-manifold and a self-dual symplectic 2-form $\theta$,  $c_1(\theta)=0$ if and only if the quotient bundle $ \Lambda^+/ \left\langle \theta \right\rangle$ is trivial. This is the case for us, with $\theta = a^i \omega_i$; it is trivialised by the projection of any pair of symplectic forms $b^i\omega_i$ and $c^i \omega_i$ for which $(a^i),(b^i),(c^i)$ are linearly independent in $\R^3$.

To upgrade the conformal structure to a metric, we need to single out a volume form. Given any positively-oriented nowhere-vanishing 4-form $\mu$, we obtain a $3\times 3$ matrix valued function $Q(\mu)$ defined by
\[
Q_{ij}(\mu) = \frac{\omega_i \wedge \omega_j}{2\mu}.
\]
Note that $Q$ is symmetric and positive definite. Up to a factor of $1/2$ it is the matrix of inner-products of the $\omega_i$ in the metric for which $\underline{\omega}$ is self-dual and for which $\mu$ is the volume form. Scaling $\mu$ will scale $Q(\mu)$ inversely. We single out the prefered volume form $\mu_{\underline{\omega}}$ by the requirement that $\det (Q(\mu_{\underline{\omega}})) =1$. We write $g_{\underline{\omega}}$ for the metric which makes $\underline{\omega}$ self-dual and has volume form $\mu_{\underline{\omega}}$. %Concretely, the metric is given by the following formula:
%\[
%g_{\underline{\omega}}(u,v) \mu_{\underline{\omega}}
%	=
%		\frac{1}{6}
%		\epsilon^{ijk}\,
%		\iota_u \omega_i \wedge \iota_v \omega_j \wedge \omega_k.
%\]

We will write $Q(\underline{\omega})$ or simply $Q$ for the matrix-valued function $Q(\mu_{\underline{\omega}}).$ An important fact is that $g_{\underline{\omega}}$ is hyperk\"ahler precisely when $Q$ is \emph{constant}. When this happens, we can apply a constant linear transformation to the $\omega_i$ to ensure that $Q_{ij} = \delta_{ij}$. Once this is done, the $\omega_i$ are a hyperk\"ahler triple, i.e.\ the K\"ahler forms associated to a quaternionic triple of complex structures which are all parallel for $g_{\underline{\omega}}$. 

We can now define the hypersymplectic flow: a time-dependent hypersymplectic triple $\underline{\omega}(t)$ is a solution of hypersymplectic flow if
\begin{equation}\label{hypersymplectic-flow}
\del_t \underline{\omega} = \diff \left( Q \,\diff^* \left( Q^{-1} \underline{\omega} \right)\right).
\end{equation}
Here, we think of $\underline{\omega}$ as a column vector of 2-forms, which is acted on by the matrix $Q^{-1}$; taking $\diff^*$ then produces a column vector of 1-forms, and so forth. It is important to note that the codifferential $\diff^*$ depends on the metric $g_{\underline{\omega}(t)}$ and so $\underline{\omega}(t)$ itself. 

Two simple remarks: if $Q$ is constant then $\underline{\omega}$ is a fixed point of the flow, since $\diff^*(Q^{-1} \underline{\omega}) = Q^{-1} \diff^*\underline{\omega} = 0$ (since $\underline{\omega}$ is closed and self-dual). Meanwhile, the right-hand side of \eqref{hypersymplectic-flow} is exact and so $[\underline{\omega}]$ is constant. These are minimum requirements to use the flow to attack Donaldson's conjecture. 

\subsection{Relationship with the $G_2$-Laplacian flow} 

The hypersymplectic flow is actually the dimensional reduction of a 7-dimensional geometric flow, called the $G_2$-Laplacian flow. We quickly recall the basic definitions, referring to \cite{Bry} for the details. In general, a 3-form on a 7-manifold $M^7$ is called a $G_2$ 3-form if the following $\Lambda^7$-valued symmetric bilinear form on $TM$ is definite:
\[
\left\langle u,v \right\rangle_\phi
=
\frac{1}{6}\iota_u \phi \wedge \iota_v \phi \wedge \phi.
\]
In other words, $\phi$ is a $G_2$ 3-form when for any nowhere-vanishing 7-form $\nu$, the symmetric 2-tensor $g_{\phi, \nu}(u,v) :=\left\langle u, v \right\rangle_\phi/\nu$ is either positive or negative definite. Asking for it to be positive orients $M$. Then we single out a distinguished positive nowhere-vanishing 7-form $\nu_\phi$ by asking that the metric $g_{\phi,\nu_\phi}$ gives $|\nu_\phi|_g = 1$, so that $\nu_\phi$ is the volume form. In this way, a $G_2$ 3-form $\phi$ completely determines an orientation and Riemannian metric on $M$, which we denote $g_\phi$. 

Up to the action of $\GL(7, \R)$ there is a unique such element of $\Lambda^3(\R^7)^*$. The stabiliser of such a 3-form is isomorphic to the exceptional Lie group $G_2$, hence the name. When $\phi$ is parallel for the Levi-Civita connection of $g_\phi$, we say that $\phi$ is torsion-free. It follows that the holonomy of $g_{\phi}$ preserves $\phi$ and so is a subgroup of $G_2 \subset \SO(7)$. One important reason to be interested in such metrics is that they are automatically Ricci flat.

Given a \emph{closed} $G_2$ 3-form $\phi$, a central question is to decide whether or not $[\phi]$ contains a torsion-free $G_2$ 3-form. It turns out that $\nabla^{g_{\phi}} \phi = 0$ is implied by the seemingly weaker conditions that $\diff \phi = 0 =\diff^*\phi$. (Note that $\diff^*$ depends on $g_\phi$ here.) With this in mind, in \cite{Bry} Bryant introduced the \emph{$G_2$-Laplacian flow}: $\del_t \phi = \Delta_{\phi} \phi$, where $\Delta_{\phi} = \diff^*\diff + \diff \diff^*$ is the Hodge Laplacian of $g_\phi$. This flow aims to deform a given closed $G_2$ 3-form into a cohomologous $G_2$ 3-form which is also coclosed and hence is torsion-free. 

Now, given a hypersymplectic structure $\underline{\omega}$ consider the 3-form $\phi$ on $\T^3 \times X$ given by
\begin{equation}
\phi = \diff t^{123} - \diff t^1 \wedge \omega_1 - \diff t^2 \wedge \omega_2 - \diff t^3 \wedge \omega_3
\label{phi-from-omega}
\end{equation}
where $t^1,t^2,t^3$ are standard coordinates on $\T^3$. The fact that $\underline{\omega}$ is hypersymplectic ensures that $\phi$ is a closed $G_2$ 3-form. The 7-dimensional and 4-dimensional metrics are related by
\[
g_\phi = Q_{ij}\diff t^i\diff t^j \oplus g_{\underline{\omega}}.
\]
(This explains the choice $\det Q = 1$ for setting the scale of the metric $g_{\underline{\omega}}$.)  In \cite{FY} it is shown that if one starts the $G_2$-Laplacian flow with a 3-form $\phi$ of the special form~\eqref{phi-from-omega}, then $\phi(t)$ has the same shape, defined by a hypersymplectic triple $\underline{\omega}(t)$ which itself evolves by the hypersymplectic flow. 

Very little is known about the general behaviour of the $G_2$-Laplacian flow. Bryant and Xu proved that there is a unique solution for small time \cite{BX}. Without symmetry assumptions, the only long-time existence result, due to Lotay and Wei, is that if $\phi$ is a torsion-free $G_2$ 3-form, and the flow is started sufficiently close to $\phi$ in $[\phi]$ then the flow will exist for all time and converge back to $\phi$ modulo diffeomorphisms \cite{LW2}. There is no known example of a finite-time singularity of the flow on a compact 7-manifold (although such examples can be found on non-compact manifolds \cite{Lau2,MOV,Nico2}).

For the hypersymplectic flow, things appear more hopeful. Firstly, note that Bryant and Xu's short-time existence result imples that for any hypersymplectic structure $\underline{\omega}$ there is a unique solution, for short time at least, to the hypersymplectic flow starting at $\underline{\omega}$. 

Secondly, as Hitchin observed~\cite{Hitchin}, the $G_2$-Laplacian flow is the gradient flow of the total volume functional. In particular the volume is increasing along the flow. When the $G_2$-structure $\phi$ comes from a hypersymplectic structure $\underline{\omega}$, there is a \emph{topological} upper-bound on the volume:
\[
\Vol(\T^3 \times X, g_\phi) = \left(2\pi\right)^3 \Vol(X,g_{\underline{\omega}}) \leq \frac{\left(2\pi\right)^3}{6} \int_X \omega_1^2 + \omega_2^2 + \omega_3^2.
\]
This is in stark contrast to the general case. There are examples, due to Mayther~\cite{Mayther}, of closed $G_2$ forms on a compact 7-manifold for which the volume is arbitrarily large. 

The main advantage the hypersymplectic flow has over the general $G_2$-Laplacian flow---and one we will exploit for this article---is the following extension criteria proved in \cite{FY}. To state it, we recall that given a closed $G_2$ form $\phi$ the \emph{torsion} is the 2-form $\mathbf{T} = -\frac{1}{2} \diff^* \phi$. For a $G_2$ structure on $\T^3 \times X$ of the form~\eqref{phi-from-omega}, the torsion has the form $\mathbf{T} =-\frac{1}{2}\diff t^i \wedge \tau_i$ for a triple of 1-forms $\tau_i$ on $X$. It follows that $2|\mathbf{T}|^2 = Q^{ij}\left\langle \tau_i,\tau_j \right\rangle$ where $Q^{ij}$ are the elements of the inverse matrix $Q^{-1}$. 

\begin{theorem}[Fine--Yao \cite{FY}]\label{FY-extension}
Let $\underline{\omega}(t)$ be a solution to the hypersymplectic flow on $X \times [0,s)$ where $X$ is a compact 4-manifold. Suppose that the scalar quantity 
\[
\cT := 2|\mathbf{T}|^2 = Q^{ij}\left\langle \tau_i,\tau_j \right\rangle
\]
is bounded on $X \times [0,s)$. Then the flow extends past $t=s$. 
\end{theorem}

%We remark that for a general compact $G_2$-Laplacian flow, the best extension result that is known is that if $|\mathbf{Rm}| + |\nabla \mathbf{T}|$ is bounded then the flow extends. This result, due to Lotay and Wei \cite{LW1}, is the starting point in the proof of Theorem~\ref{FY-extension}.

\subsection{$\T^3$-invariant hypersymplectic structures on $\T^4$ and the main result.}\label{introduction-main-result}

In this paper we study a certain class of $\T^3$-invariant hypersymplectic structures on $\T^4$. We use the standard action of $\T^3$ on $\T^4$ given by
\begin{equation}
\label{def:T3action}
\left(e^{it_1}, e^{it_2}, e^{it_3}\right)\cdot \left( e^{ix_0}, e^{ix_1}, e^{ix_2}, e^{ix_3}\right) = \left( e^{ix_0}, e^{i(t_1+ x_1)}, e^{i(t_2+x_2)}, e^{i(t_3+x_3)}\right). 
\end{equation}
Lemma~\ref{NF-01} below shows that given any $\T^3$-invariant hypersymplectic triple, we can act by a constant linear transformation on the components to put it in the form
\begin{equation}
\omega_i = \alpha_{ij}(x_0) \diff x_0 \wedge \diff x_j + \frac{1}{2}\epsilon_{ipq} \diff x_{p} \wedge \diff x_q
\label{normal-form-equation}
\end{equation}
where $a_{ij} \colon \T \to \R$ are functions of a single variable.
\begin{definition}
A $\T^3$-invariant hypersymplectic structure of the form~\eqref{normal-form-equation} is said to be in \emph{normal form}. If, moreover, $\alpha_{ij} = \alpha_{ji}$ we say the structure is in \emph{symmetric normal form}.
\end{definition}

We can now state the main result of this article.

\begin{theorem}\label{main-result}
Let $\underline{\omega}$ be a $\T^3$-invariant hypersymplectic structure on $\T^4$ which is in symmetric normal form. The hypersymplectic flow $\underline{\omega}(t)$ starting at $\underline{\omega}$ exists for all $t \in [0,\infty)$. Moreover, there exists a path of diffeomorphisms $G(t) \colon \T^4 \to \T^4$ starting at the identity, such that $G(t)^*\underline{\omega}(t)$ converges as $t\to \infty$ to the unique flat hyperk\"ahler metric in~$[\underline{\omega}]$.

In particular, Donaldson's Conjecture~\ref{Donaldson-conjecture} holds for $\underline{\omega}$ and we have a positive answer to Question~\ref{torus-question} for the symplectic forms $a^i\omega_i$ on $\T^4$. 
\end{theorem}

This is a generalisation of \cite{HWY}, which proved the same convergence result for $\T^3$-invariant hypersymplectic structures of \emph{simple type}, i.e. those of the very special form where $a_{ij}= \delta_{ij} (1+\phi_i'')$ is diagonal, with entries determined by three functions $\phi_i \colon \T \to \R$  \cite{HWY}. By contrast, in Theorem~\ref{main-result} the eigendirections of $\alpha_{ij}(x_0)$ can depend on $x_0$. 

Theorem~\ref{main-result} and the main result of \cite{HWY} which it subsumes are the only convergence results currently known for the hypersymplectic flow, besides those which merely invoke the dynamic stability of Lotay--Wei \cite{LW2} by starting the flow very close to a hyperk\"ahler triple.

From a purely PDE perspective, Theorem~\ref{main-result} can be seen as a convergence result for a non-linear second-order evolution equation for a symmetric positive-definite matrix-valued function $\alpha$ of two variables $(x_0,t)$. Explicitly, the equation is: 
\[
\del_t \alpha = \left( \frac{1}{V} \left( \frac{\alpha}{V}\right)' \right)'
\]
where $V = (\det \alpha)^{1/3}$. (This equation is derived in Lemma~\ref{flow-as-a-pde}.) It turns out that this system is \emph{not} parabolic (see Remark~\ref{not-parabolic}). This is perhaps unsurprising because we began with a flow which was only parabolic modulo diffeomorphisms. It means, however, that analytic methods must be supplemented by geometric arguments in the proof.  

\subsection{Acknowledgments}

We would like thank Song Sun for some very helpful conversations. JF is supported by the ``Excellence of Science'' grant number 4000725 and the FNRS grant PDR T.0082.21.  

 \section{$\T^3$-invariant hypersymplectic structures}

\subsection{Normal form}

We begin by putting $\T^3$-invariant hypersymplectic structures into ``normal form''~\eqref{normal-form-equation} as described in the introduction. We first set out our notation and conventions. Recall we use coordinates $(x_0,x_1,x_2,x_3)$ on $\T^4$ in which the standard $\T^3$-action given in~\eqref{def:T3action} rotates $(x_1,x_2,x_3)$. We write $\diff x_{ij}$ for the 2-form $\diff x_i \wedge \diff x_j$, and similarly for higher degree forms. We assume that our coordinates $x_1,x_2,x_3$ have been ordered so that the orientation induced by the hypersymplectic structure we're working with makes $\diff x_{0123}$ a positive 4-form. We use the summation convention that repeated indices are summed over 1,2,3. 

\begin{lemma}\label{invariant-closed-2-forms}
Any $\mathbb{T}^3$-invariant closed 2-form $\theta$ on $\mathbb{T}^4$ has the form 
\[
\theta = a_{i}(x_0) \diff x_{0i} + \eta_{pq} \diff x_{pq}
\]
for functions $a_i \colon \T^1 \to \R$ and a constant skew symmetric matrix $\eta_{pq}$. 
\end{lemma}
\begin{proof}
Any 2-form on $\T^4$ has the given form where $a_i, \eta_{pq} \colon \T^4 \to \R$ are functions of all four variables. $\T^3$-invariance forces these coefficient functions to depend only on $x_0$, and now $\diff \theta=0$ forces $\eta_{pq}$ to be constant. 
\end{proof}

Now let $\underline{\omega}$ be a $\T^3$-invariant hypersymplectic structure, and $\underline{\omega}(t)$ the hypersymplectic flow starting at $\underline{\omega}$. Since the solution is unique, $\underline{\omega}(t)$ remains $\T^3$-invariant. Using Lemma~\ref{invariant-closed-2-forms}, we write
\[
\omega_i(t)=
	a_{ij}(x_0,t) \diff x_{0j} + \eta_{i,pq}(t) \diff x_{pq}.
\]
\begin{lemma}\label{eta-does-not-change}
$\eta_{i,pq}(t) = \eta_{i,pq}(0)$ is independent of time. 
\end{lemma}
\begin{proof}
This follows from the fact that $[\underline{\omega}(t)] = [a_{ij}(x_0,t)\diff x_{0j}] + \eta_{i,pq}(t)[\diff x_{pq}]$ is independent of $t$. 
\end{proof}

\begin{lemma}[Normal form]\label{NF-01}
Let $\underline\omega=\left(\omega_1, \omega_2, \omega_3\right)$ be a  $\T^3$-invariant hypersymplectic structure on $\T^4$. The exists an invertible constant matrix $A$ such that $\widetilde{\omega}_j = A_{ij} \omega_i$ satisfies
	\[
	\widetilde\omega_i = \alpha_{ij}(x_0) dx_{0j} + \frac{1}{2}\epsilon_{ipq}\diff x_{pq},
	\]
i.e., $\widetilde{\omega}$ is in normal form, 
\end{lemma}

\begin{proof}
Write $\omega_i = dx_0\wedge \alpha_i + \eta_i$ where $\eta_i$ is a constant 2-form and $\alpha_i$ is a 1-form in $\Span\{\diff x_1, \diff x_2, \diff x_3\}$. Since $\underline{\omega}$ is hypersymplectic, given any $\xi\in \R^3\setminus \{0\}$, 
	\[
	\frac{ \xi^i \omega_i \wedge \xi^j\omega_j}{dx_{0123}}
	=
	\frac{ 2\xi^i \alpha_i\wedge \xi^j\eta_j }{dx_{123}}>0. 
	\]
In particular, $\xi^j \eta_j\neq 0$. This means that the $\eta_j$ are a frame of $\Lambda^2\T^3$. So there is a unique invertible matrix $A$ such that $A_{ij} \eta_j = \frac{1}{2}\epsilon_{ipq} \diff x_{pq}$. This matrix will transform $\underline\omega$ to be in the normal form of the statement. 
\end{proof}

\begin{lemma}\label{normal-form-preserved}
If $\underline{\omega}(t)$ is a solution to the hypersymplectic flow which starts at $\underline{\omega}(0)$ which is in normal form~\eqref{normal-form-equation}, then $\underline{\omega}(t)$ remains in normal form for all $t$. 
\end{lemma}

\begin{proof}
This follows directly from Lemma~\ref{eta-does-not-change}.
\end{proof}

Since the condition of being in normal form is preserved under the flow, it suffices to consider triples in normal form.

\subsection{The skew-symmetric part of the coefficient matrix}

From here on we work with a hypersymplectic flow $\underline{\omega}(t)$ starting from a triple in normal form~\eqref{normal-form-equation}. 

By Lemma~\ref{normal-form-preserved}, $\underline{\omega}(t)$ is in normal form for all $t$. We split the coefficient matrix into symmetric and skew-symmetric parts, $\alpha(t) = \beta(t) + \gamma(t)$ where $\beta(t) = \beta(t)^T$ and $\gamma(t) = - \gamma(t)^T$. The main result of this subsection is that under the hypersymplectic flow, $\del_t\gamma =0$ (Proposition~\ref{symmetric-preserved} below). Along the way we also collect some useful formulae.

\begin{lemma}\label{Q-V-formulae}
Let $\underline{\omega}$ be a hypersymplectic structure in normal form~\eqref{normal-form-equation}. Write $\alpha = \beta + \gamma$ where $\beta^T =\beta$ and $\gamma^T = -\gamma$. Then $\det \beta >0$. Letting $V= (\det \beta)^{1/3}$ we have
\begin{enumerate}
\item \label{volume}
$
\mu_{\underline{\omega}} = V \diff x_{0123}
$.
\item
$Q = V^{-1} \beta$.
\item \label{Q-matrix}
The metric tensor is
\[
g_{\underline{\omega}}
	= 
		V^{-1} \left(
		\det \alpha \diff x_0^2 
		- 
		\vec{\beta}_i\cdot \vec{\gamma}\, \diff x_0\otimes \diff x_i 
		- 
	    	\vec{\gamma}\cdot \vec{\beta}_i\, \diff x_i\otimes \diff x_0 
    		+
    		\beta_{ij}\diff x_i\otimes \diff x_j
		\right) 
\]
where $\vec{\beta}_i=\left(\beta_{i1}, \beta_{i2}, \beta_{i3}\right)^T$  and $\vec{\gamma}=\left(\gamma_{23}, \gamma_{31}, \gamma_{13}\right)^T$. 
\end{enumerate}
\end{lemma}

\begin{proof}
We compute
\begin{align*}
\omega_i \wedge \omega_j
	&=
		 \left(\alpha_{ir} \diff x_{0r} +\frac{1}{2} \epsilon_{ipq} \diff x_{pq}\right)
		 \wedge
		 \left(\alpha_{jr} \diff x_{0r} + \frac{1}{2}\epsilon_{jpq} \diff x_{pq}\right),\\
	&=
		\frac{1}{2}(\alpha_{ir} \epsilon_{jpq}+ \alpha_{jr} \epsilon_{ipq})\epsilon_{rpq} 
		\diff x_{0123},\\
	&=
		(\alpha_{ij} + \alpha_{ji}) \diff x_{0123},\\	
	&=
		2\beta_{ij} \diff x_{0123}.
\end{align*}
Since $\underline{\omega}$ is hypersymplectic, and $\diff x_{0123}$ is positive (by hypothesis on the ordering of our coordinates $x_1, x_2, x_3$), it follows that $\beta_{ij}$ is positive definite. In particular $\det \beta >0$. Now with $V = (\det \beta)^{1/3}$ we see that 
\[
\frac{\omega_i \wedge \omega_j}{2V \diff x_{0123}} = V^{-1} \beta
\]
has determinant equal to 1. Parts~\ref{volume} and~\ref{Q-matrix} follow immediately.  

To compute the matrix tensor we use the identity
\[
g_{\underline{\omega}}(u,v) \mu_{\underline{\omega}}
	=
		\frac{1}{6} \epsilon_{ijk} \iota_u \omega_i \wedge \iota_v \omega_j \wedge \omega_k.
\]
(See the second section of~\cite{FY}.) From here we have the folllowing formulae for the coefficients of $g_{\underline{\omega}} = g_{ab} \diff x_a \otimes \diff x_b$:
\begin{align*}
V g_{00}  
	&=
		\frac{1}{6} \alpha_{ir}\alpha_{js} \alpha_{kt} \epsilon_{ijk}\epsilon_{rst}\\
	&=
		\det \alpha,\\
Vg_{01}
	&=
		\frac{1}{2}\left( \alpha_{11} (\alpha_{32} - \alpha_{23}) + \alpha_{21}\alpha_{13} - \alpha_{31} \alpha_{12}\right)\\
	&=
		\beta_{11} \gamma_{32} + \beta_{12} \gamma_{13} + \beta_{13}\gamma_{21}\\
	&=
		-\vec{\beta}_1\cdot  \vec{\gamma},
\end{align*}	
with similar formulae for $Vg_{02}$ and $Vg_{03}$. This gives the $\diff x_0^2$ and $\diff  x_0 \otimes \diff x_i$ coefficients. Finally, one checks that $V g_{ij} = \frac{1}{2}( \alpha_{ij} + \alpha_{ji}) = \beta_{ij}$.
\end{proof}
It turns out to be more convenient to change from the ``standard'' coframe $\diff x_0, \diff x_1, \diff x_2, \diff x_3$ to one adapted to the hypersymplectic structure. 

\begin{definition}
Let $\underline{\omega}$ be a $\T^3$-invariant hypersymplectic structure in normal form~\eqref{normal-form-equation}. The coframe of $T^*\T^4$ associated to $\underline{\omega}$ is 
    \[
    \vartheta_0 = dx_0, \;
    \vartheta_1 = dx_1 -\gamma_{23}dx_0,\; 
    \vartheta_2 = dx_2 -\gamma_{31} dx_0,\;
    \vartheta_3 = dx_3 - \gamma_{12} dx_0,
    \]
    where $\gamma$ is the skew part of the coefficient matrix $\alpha$. 
\end{definition}

\begin{lemma}\label{lemma-metric-in-coframe}
With respect to the coframe $\vartheta_a$ associated to $\underline{\omega}$, the metric tensor is
\begin{equation}
g_{\underline{\omega}} = V^2 \vartheta_0^2 + Q_{ij} \vartheta_i \otimes \vartheta_j.
\label{metric-in-coframe}
\end{equation}
\end{lemma}

\begin{proof}
The dual framing of $T\T^4$ is
\[
	\frac{\partial}{\partial x_0} 
   	+ 
	\gamma_{23}\frac{\partial}{\partial x_1} 
	+ 
	\gamma_{31}\frac{\partial}{\partial x_2} 
	+ 
	\gamma_{12}\frac{\partial}{\partial x_3}, \ 
	\frac{\partial}{\partial x_1},\ 
	\frac{\partial}{\partial x_2},\ 
	\frac{\partial}{\partial x_3}.
\]
From here one checks that
\[
g_{\underline{\omega}}
	=
		V^{-1} \left(\det \alpha - \vec{\gamma}^T \beta \vec{\gamma}\right)
		+
		Q_{ij} \vartheta_i \otimes \vartheta_j.
\]
(Recall that $Q_{ij} = V^{-1} \beta_{ij}$.) The result now follows from the following lemma about $3\times 3$ determinants.
\end{proof}

\begin{lemma}
\label{lem:linear-algebra}
Let $\alpha$ be a $3\times 3$ matrix, $\beta$ and $\gamma$ be its symmetric and skew-symmetric parts, i.e. $\beta=\frac{1}{2}(\alpha+\alpha^T)$ and $\gamma=\frac{1}{2}(\alpha-\alpha^T)$. Then 
\[
\det\alpha = \det\beta + \vec{\gamma}^T\beta \vec{\gamma}. 
\]
\end{lemma}

\begin{proof}
Choose $P\in \SO(3)$ such that $P\beta P^T=\text{diag} (\lambda_1, \lambda_2, \lambda_3)$. Then the skew matrix $\tilde\gamma:= P\gamma P^T$ has associated vector $\vec{\tilde{\gamma}} = (\tilde{\gamma}_{23}, \tilde{\gamma}_{31}, \tilde{\gamma}_{12})^T$ which satisfies $\vec{\tilde{\gamma}} = P \vec{\gamma}$. One can either compute this directly, or note that it is a restatement of the fact that the following diagram 
\[
\begin{tikzcd}
\mathfrak{so}(3) \arrow{r}{\Phi}\arrow{d}{\ad_P}&\R^3\arrow{d}{P}\\
\mathfrak{so}(3) \arrow{r}{\Phi} & \R^3
\end{tikzcd}
\]
commutes for any $P\in \SO(3)$, i.e. the natural isomorphism $\Phi: \mathfrak{so}(3)\to \R^3$ is equivariant with respect to the $\SO(3)$ actions. 

As a consequence, 
\begin{align*}
\det\alpha
& =\det\left(P\alpha P^T\right) = \det\left(P\beta P^T+P\gamma P^T\right)
=
\left|
\begin{array}{ccc}
\lambda_1 & \tilde \gamma_{12} & \tilde \gamma_{13}\\
\tilde\gamma_{21} & \lambda_2 & \tilde \gamma_{23}\\
\tilde\gamma_{31} & \tilde\gamma_{32} & \lambda_3
\end{array}
\right|\\
& = 
\lambda_1\lambda_2\lambda_3 + \lambda_1\tilde\gamma_{23}^2 + \lambda_2\tilde\gamma_{31}^2 + \lambda_3\tilde\gamma_{12}^2\\
& = 
\det\beta + \vec{\tilde\gamma}^T \text{diag}\left(\lambda_1,\lambda_2,\lambda_3\right) \vec{\tilde\gamma}\\
& = \det\beta + \vec{\gamma}^T\beta \vec{\gamma}.
\qedhere
\end{align*}
\end{proof}

\begin{lemma}\label{some-coframe-formulae}
Using the coframe $\vartheta_a$ associated to $\underline{\omega}$, we have the following formulae:
\begin{enumerate}
\item \label{omega-in-coframe}
$ \omega_i = \beta_{ip} \vartheta_0 \wedge \vartheta_p + \frac{1}{2}\epsilon_{ijk} \vartheta _j \wedge \vartheta_k$.
\item 
$*\vartheta_0= V^{-1}\vartheta_1\wedge\vartheta_2\wedge\vartheta_3$.
\item  
$ *\vartheta_p = -VQ^{pi}\vartheta_0\wedge\vartheta_j\wedge\vartheta_k,\quad (ijk)=(123)$.
 \item \label{3-form-Hodge-star}
 $*\left(\vartheta_0\wedge\vartheta_j\wedge\vartheta_k\right) = V^{-1}Q_{ip}\vartheta_p, \quad (ijk)=(123)$.
\item\label{torsion-forms-in-coframe}
The torsion $1$-forms $\tau = Q\diff^*(Q^{-1}\underline{\omega})$ are 
    \[
    \tau_i = V^{-1}Q_{ij}'\vartheta_j,\;\; i=1,2,3,
    \]
    where the prime denotes derivative with respect to $x_0$.
\end{enumerate}
\end{lemma}
\begin{proof}
The formula for $\omega_i$ in terms of the coframe $\vartheta_a$ is a direct calculation. The formulae for the various Hodge stars follow from the expression~\eqref{metric-in-coframe} for $g_{\underline{\omega}}$. To obtain the formula for the $\tau_i$, note that
\begin{align*}
\tau_i 
	&=
		- Q_{ij} * \diff (Q^{jp} \omega_p)\\
	&=
		Q^{kp} Q'_{ik} * (\diff x_0 \wedge \omega_p)\\
	&=
		\frac{1}{2}Q^{kp} Q'_{ik} \epsilon_{pqr} * (\vartheta_0 \wedge \vartheta_q \wedge \vartheta_r)\\
	&=
		\frac{1}{2}V^{-1}  \epsilon_{pqr}\epsilon_{qrs} Q^{kp}Q'_{ik}Q_{st} \vartheta_t\\
	&=
		V^{-1} Q'_{ik} \vartheta_k,
\end{align*}
where in the second line we have used that $\diff \omega_p=0$, in the third line we've used the formula~\eqref{omega-in-coframe} for $\omega_p$, in the fourth line we've used~\eqref{3-form-Hodge-star} and in the final line we've used the identity $\epsilon_{pqr}\epsilon_{sqr} = 2 \delta_{sp}$.
\end{proof}

\begin{proposition}\label{symmetric-preserved}
Let $\underline{\omega}$ be a $\T^3$-invariant hypersymplectic structure on $\T^4$, in normal form~\eqref{normal-form-equation} with coefficient matrix $\alpha$. Under the hypersymplectic flow, the skew-symmetric part of $\alpha$ is independent of time. In particular, if $\alpha$ is symmetric at $t=0$, then this is true along the flow, for as long as it exists. 
\end{proposition}

\begin{proof}
By Lemma~\ref{normal-form-preserved} we know that $\underline{\omega}(t)$ is in normal form for all $t$. We write its coefficient matrix $\alpha(t) = \beta(t) + \gamma(t)$ where $\beta(t) = \frac{1}{2} (\alpha(t) + \alpha(t)^T)$ and $\gamma(t) = \frac{1}{2}(\alpha(t) - \alpha(t)^T)$. We denote $\vartheta_i(t) = \diff x_i - \frac{1}{2} \epsilon_{ijk} \gamma_{jk}(t) \diff x_0$ for the coframe associated to $\underline{\omega}(t)$. Note that  $\vartheta_0= \diff x_0$ is automatically independent of $t$, whilst
\[
\del_t \vartheta_i = - \frac{1}{2} \epsilon_{ijk} (\del_t \gamma_{jk}) \vartheta_0.
\]

By part~\ref{omega-in-coframe} of Lemma~\ref{some-coframe-formulae}, we have
\[
\omega_i(t)
	=
		\beta_{ip}(t) \vartheta_0 \wedge \vartheta_p(t) 
		+ 
		\frac{1}{2}\epsilon_{ijk} \vartheta_j(t) \wedge \vartheta_k(t).
\]
So
\begin{align*}
\del_t \omega_i
	&=
		\del_t \beta_{ip} \vartheta_0 \wedge \vartheta_p
		-
		\frac{1}{4} \epsilon_{ijk} (
		\epsilon_{jrs} \del_t \gamma_{rs} \vartheta_0 \wedge \vartheta_k
		+
		\epsilon_{krs}\del_t \gamma_{rs} \vartheta_j \wedge \vartheta_0)\\
	&=
		\left(
		\del_t \beta_{ip}
		-
		\frac{1}{4} (
			\epsilon_{ijp} \epsilon_{jrs} 
			-
			\epsilon_{ipk} \epsilon_{krs}) 
			\del_t \gamma_{rs}
		\right)
		\vartheta_0 \wedge \vartheta_p\\
	&=
		\left(
		\del_t \beta_{ip}
		-
		\frac{1}{2} (
			\delta_{pr}\delta_{is}- \delta_{ps}\delta_{ir}
			)
			\del_t \gamma_{rs}
		\right)
		\vartheta_0 \wedge \vartheta_p\\
	&=
		\left(
		\del_t \beta_{ip} + \del_t \gamma_{ip}
		\right)
		\vartheta_0 \wedge \vartheta_p.
\end{align*}
Meanwhile, by definition of the hypersymplectic flow,
\[
\del_t \omega_i 
	= 
		\diff \tau_i
	=
		\left( V^{-1} Q'_{ip}\right)' \vartheta_0 \wedge \vartheta_p,
\]
where we have taken $\diff$ of the formula of part~\ref{torsion-forms-in-coframe} of Lemma~\ref{some-coframe-formulae}. We conclude that
\[
\del_t \beta_{ij} + \del_t \gamma_{ij} 
	=
		\left( V^{-1} Q'_{ij}\right)'	.
\]
The right-hand side is symmetric, and so the skew part of the left-hand side must vanish. In other words, $\del_t \gamma = 0$ as claimed.
\end{proof}

\subsection{The case of a symmetric coefficient matrix}

From now on, we focus entirely on the case of a $\T^3$-invariant hypersymplectic structure in normal form~\eqref{normal-form-equation} in which the coefficient matrix $\alpha_{ij}$ is \emph{symmetric}. 

\begin{lemma}\label{metric-when-symmetric}
Let $\underline{\omega}$ be a $\T^3$-hypersymplectic form in normal form~\eqref{normal-form-equation}, with symmetric coefficient matrix $\alpha \colon \T^1 \to S^2 \R^3$. The metric is given by
\[
g_{\underline{\omega}(t)} = V^2 \diff x_0^2 + Q_{ij} \diff x_i \diff x_j.
\]
\end{lemma}

\begin{proof}
When $\alpha$ is symmetric, the coframe $\vartheta_a$ of $T^*\T^4$ which featured in the previous subsection is simply the coordinate coframe $\vartheta_a = \diff x_a$. The result now follows from~\eqref{metric-in-coframe}
\end{proof}

\begin{lemma}\label{flow-as-a-pde}
Let $\underline{\omega}$ be a $\T^3$-hypersymplectic form in normal form~\eqref{normal-form-equation}, with symmetric coefficient matrix $\alpha \colon \T^1 \to S^2 \R^3$. The hypersymplectic flow starting at $\underline{\omega}$ is equivalent to
\begin{equation}
\del_t \alpha_{ij}
	=
		\left( \frac{1}{V}\left( \frac{\alpha_{ij}}{V}\right)'\right)',
\label{alpha-evolution}
\end{equation}
where $V = (\det \alpha)^{1/3}$ and the prime denotes $\diff/\diff x_0$. 
\end{lemma}
\begin{proof}
When $\alpha$ is symmetric, Proposition~\ref{symmetric-preserved} tells us the flow remains in symmetric normal form. From Lemma~\ref{Q-V-formulae}, $Q = V^{-1}\alpha$ along the flow. Since $\alpha$ is symmetric, $\vartheta_a = \diff x_a$; so Lemma~\ref{some-coframe-formulae} gives $\tau_i = V^{-1} Q'_{ij} \diff x_j$ and the result follows from the definition $\del_t \omega_i = \diff \tau_i$ of the hypersymplectic flow.
\end{proof}

\begin{remark}\label{not-parabolic}
The system \eqref{alpha-evolution} is a nonlinear system of PDE on $\mathbb{T}^1\times [0,+\infty)$ for $6$ functions. It is interesting to look at this system from a purely analytic point of view (as is done in \cite{HWY}). Expanding out the derivatives, the system is 
\begin{equation}
\partial_t \alpha_{ij}
= 
\frac{1}{V^2} 
\left( \alpha'' - \frac{1}{3}\langle \alpha,\alpha''\rangle_\alpha \alpha\right) 
- 
\frac{1}{V^2}\langle \alpha,\alpha'\rangle_\alpha \alpha'
+ 
\frac{1}{3V^2}\langle \alpha',\alpha'\rangle_\alpha \alpha
+ 
\frac{2}{9V^2} \langle \alpha,\alpha'\rangle_\alpha^2 \alpha,
\label{flow-expanded}
\end{equation}
where $\langle\beta,\gamma\rangle_\alpha=\tr \left( \alpha^{-1}\beta\alpha^{-1}\gamma\right)$. This is  the Riemmanian metric of the symmetric space metric on $S^2_+\R^3$ of positive definite inner-products on $\R^3$: we treat this as an open set in the vector space $S^2\R^3$ of all symmetric bilinear forms; then the metric evaluated on $\beta, \gamma \in S^2 \R^3 =T_\alpha S^2_+ \R^3$ is $\langle\beta,\gamma\rangle_\alpha$.

We linearise the right-hand side of~\eqref{flow-expanded}, which is denoted by $\mathscr{D}(\alpha)$, at $\alpha$. This gives the second-order linear differential operator $L_\alpha \colon \Gamma(\T^1, S^2 \R^3) \to \Gamma(\T^1, S^2\R^3)$ where
\begin{equation}
\begin{split}
L_\alpha(\beta)
& = 
\frac{1}{V^2}
\Big[ 
\beta'' 
- 
\frac{1}{3}\langle \alpha,\beta''\rangle_\alpha\alpha 
- 
\frac{1}{3}\langle \alpha,\alpha''\rangle_\alpha \beta 
-\langle\alpha,\beta'\rangle_\alpha\alpha' - \langle\alpha,\alpha'\rangle_\alpha\beta'\\
& \qquad \quad\qquad
+ 
\frac{2}{3}\langle\alpha',\beta'\rangle_\alpha\alpha 
- 
\frac{2}{3}\langle \alpha'\alpha^{-1}\alpha',\beta\rangle_\alpha\alpha 
+ 
\frac{1}{3}\langle\alpha',\alpha'\rangle_\alpha\beta\\
& \qquad\quad\qquad\qquad
+ 
\frac{4}{9}\langle\alpha,\alpha'\rangle_\alpha \langle\alpha,\beta'\rangle_\alpha\alpha
+ 
\frac{2}{9}\langle \alpha,\alpha'\rangle_\alpha^2\beta
\Big]
 - 
\frac{2}{3}\langle\alpha,\beta\rangle_\alpha \mathscr{D}(\alpha).
\end{split}
\end{equation}
The principal symbol of $L_\alpha$ in the direction $\xi$ is given by 
\[
\sigma(\xi)(\beta) 
= 
\frac{\xi^2}{V^2}\left(\beta - \frac{1}{3}\langle\alpha,\beta\rangle_\alpha\alpha\right). 
\]
This shows that $L_\alpha$ is \emph{not} elliptic, since $\sigma(\xi)(\alpha)=0$.
\end{remark}

When $\underline{\omega}$ is in symmetric normal form, the evolution equations for $V$ and $Q$ have an important simplification over the general case. We first give a simple formula for the Laplacian, and then derive the evolution equations.
 
\begin{lemma}\label{invariant-laplacian-formula}
Let $\underline{\omega}$ be a $\T^3$-invariant hypersymplectic structure in symmetric normal form~\eqref{normal-form-equation} and let $f \colon \T^4 \to \R$ be a smooth $\T^3$-invariant function. The Laplacian of $f$ with respect to the metric $g_{\underline{\omega}}$ is
\[
\Delta f = V^{-1}(V^{-1} f' )',
\]
where the prime denotes $\diff/\diff x_0$.
\end{lemma}
\begin{proof}
This follows from Lemmas~\ref{lemma-metric-in-coframe} and~\ref{some-coframe-formulae} which give, in particular, that $\mu_{\underline{\omega}} = V \diff x_{0123}$ and $* \diff x_0 = V^{-1}\diff x_{123}$. So
\[
\Delta f = - \diff^* \diff f = * \diff * (f' \diff x_0) = * \diff (V^{-1} f' \diff x_{123}) = (V^{-1}f')' * \diff x_{0123} = V^{-1}(V^{-1}f')'.
\qedhere
\]
\end{proof}

\begin{proposition}
Let $\underline{\omega}(t)$ be the hypersymplectic flow starting from a $\T^3$-invariant hypersymplectic structure in symmetric normal form~\eqref{normal-form-equation}. Then $V$ and $Q$ evolve as follows.
\begin{align}
\del_t V 
	&=
		\frac{1}{3} \cT V,
		\label{dV/dt}\\
\del_t Q_{ij}
	&=
		\Delta Q_{ij} - \frac{1}{3} \cT Q_{ij},
		\label{dQ/dt}
\end{align} 
where $\cT = Q^{ij}\left\langle  \tau_i, \tau_j\right\rangle$ is twice the norm-squared of the torsion of the $G_2$-structure on $\T^7$ given by~\eqref{phi-from-omega}. 
\end{proposition}
\begin{proof}
In the case of an arbitrary hypersymplectic flow, the evolution equations for $\mu_{\underline{\omega}}$ and $Q$ are given in \cite{FY}:
\begin{align}
\del_t \mu_{\underline{\omega}}
	&=
		\frac{1}{3}\cT \mu_{\underline{\omega}}. 
		\label{dmu/dt}\\
\del_t Q_{ij}
	&=
		\Delta Q_{ij} - \left\langle \diff Q_{ip}, Q^{pq} \diff Q_{qj} \right\rangle
		+ \left\langle \tau_i, \tau_j \right\rangle
		- \frac{1}{3} \cT Q_{ij}.
		\label{general-dQ/dt}
\end{align}
(Equation~\eqref{dmu/dt} is (44) in \cite{FY}; in the notation of \cite{FY}, $|\mathbf{T}|^2 = \frac{1}{2}\cT$; equation~\eqref{general-dQ/dt} is Corollary 4.2 of \cite{FY}.) Now~\eqref{dV/dt} follows from~\eqref{dmu/dt} and the fact that $\mu_{\underline{\omega}} = V \diff x_{0123}$. Meanwhile~\eqref{dQ/dt} follows from~\eqref{general-dQ/dt} and the fact that
\[
\left\langle \tau_i,\tau_j \right\rangle
=
\left\langle V^{-1}Q'_{ik} \diff x_k, V^{-1} Q'_{jl} \diff x_l \right\rangle
=
V^{-2} Q'_{ik}Q^{kl}Q_{lj}
=
\left\langle \diff Q_{ik}, Q^{kl} \diff Q_{lj} \right\rangle,
\]
where we have used that the metric on $\T^4$ is $g_{\underline{\omega}} = V^2 \diff x_0^2 + Q_{ij}\diff x_i \diff x_j$ (Lemma~\ref{lemma-metric-in-coframe}) to compute $\left\langle \diff x_k, \diff x_l \right\rangle = Q^{kl}$.

Alternatively these two equations can be computed directly. For~\eqref{dV/dt}, use $\log V = \frac{1}{3} \log\det \alpha$ to obtain
\[
\frac{1}{V} \del_t V 
= 
\frac{1}{3} \alpha^{ij} \del_t\alpha_{ij}
=
\frac{1}{3V} Q^{ij}\left( \frac{1}{V} Q'_{ij}\right)'
=
\frac{1}{3V^2}Q^{ij} \left( Q''_{ij} - \frac{V'}{V} Q'_{ij}\right),
\]
where we've used $Q = V^{-1} \alpha$ and the evolution equation~\eqref{alpha-evolution} for $\alpha$ in the form
\[
\del_t \alpha = \left( \frac{1}{V} \left( \frac{\alpha}{V} \right)'\right)' = \left( \frac{1}{V} Q'\right)'.
\]
Now differentiating the identity $\det Q \equiv 1$ implies that $Q^{ij} Q'_{ij}=0$ and hence, differentiating again, $Q^{ij}Q''_{ij} - Q^{ip}Q'_{pq}Q^{qj}Q'_{ij} = 0$. From here, and using the fact that $\cT = V^{-2} Q^{ip}Q'_{pq}Q^{qj}Q'_{ij}$ we obtain~\eqref{dV/dt}.

For~\eqref{dQ/dt}, we have
\[
\del_t Q 
= 
	\del_t (V^{-1} \alpha) 
= 
	\frac{1}{V} \del_t \alpha - \alpha \frac{1}{V^2} \del_t V
=
	\frac{1}{V}\left( \frac{1}{V} Q' \right)' 
	-
	\frac{1}{3V} \cT \alpha
=
	\Delta Q - \frac{1}{3} \cT Q,
\]
where in the second equality we use the evolution~\eqref{alpha-evolution} of $\alpha$ and in the final step we have used Lemma~\ref{invariant-laplacian-formula} to recognise $\Delta Q$.
\end{proof}

\subsection{Diagonal coefficient matrix}

We make a short digression to discuss the $\T^3$-invariant hpyersymplectic structures of ``simple type'', introduced in \cite{HWY}. These are in normal form with \emph{diagonal} coefficient matrix  $\alpha_{ij} = \delta_{ij} (1+ \phi_i'')$ for some potential functions $\phi_i \colon \T^1 \to \R$. In that paper, it was not shown directly that simple type structures are preserved under the flow. Instead the evolution equations of the $\phi_i$ were derived \emph{assuming} this, and then short-time existence of the resulting system was proved after the fact. This was made difficult because these evolution equations are not parabolic (just as is described in Remark~\ref{not-parabolic}). Instead, the system was converted into a combined differential-integral system; short-time existence for this was then shown using the same functional analytic ideas familiar from the standard theory of parabolic equations. 

In this subsection, we will show directly that if we begin with an off-diagonal term in $\alpha$ which is zero at $t=0$, then this condition is automatically preserved under the flow. In particular, hypersymplectic stuctures of simple type are automatically preserved along the flow. This gives a much simpler argument for short-time existence of the flow considered in~\cite{HWY} (although the approach here relies on Bryant--Xu's short-time existence result for the $G_2$-Laplacian flow~\cite{BX}).

\begin{proposition}
Let $\underline{\omega}(t)$ be a hypersymplectic flow on $\T^4$, starting from a $\T^3$-invariant structure in symmetric normal form. Suppose that $\alpha_{12}(\cdot, 0) \equiv 0$. Then $\alpha_{12}(\cdot, t) \equiv 0$ for as long as the flow exists. The same is true for the other off-diagonal terms, $\alpha_{23}$ and $\alpha_{13}$. In particular, if $\alpha(\cdot, 0)$ is diagonal then this remains true for as long as the flow exists.
\end{proposition}
\begin{proof}
Let $\T^2_{02}$ denote the 2-torus in the coordinate directions $(x_0,x_2)$, oriented by $\diff x_0 \wedge \diff x_2$. Since $[\omega_1(t)]$ is constant,
\[
0 
	= 
		\frac{\diff}{\diff t} \int_{\T^2_{02}} \omega_1
	=
		2\pi \frac{\diff }{\diff t} \int_{0}^{2\pi} \alpha_{12}\, \diff x_0.
\]
So
\[
\int_0^{2\pi} Q_{12}(x_0,t) V(x_0t) \diff x_0
	=
		\int_0^{2\pi} \alpha_{12}(x_0,t) \diff x_0
	=
		\int_0^{2\pi} \alpha_{12}(x_0,0) \diff x_0
	=
		0.
\]
In particular, the spatial maximum of $Q_{12}$ is non-negative:
\[
\overline{Q}_{12}(t) := \max_{x_ \in \T^1} Q_{12}(x_0,t) \geq 0.
\]
Now by~\eqref{dQ/dt}, $\del_t Q_{12} = \Delta Q_{12} - \frac{1}{3} \cT Q_{12}$. But at a spatial maximum of $Q_{12}$ the term $-\frac{1}{3}\cT Q_{12}$ is non-positive. It follows from the maximum principle that $\overline{Q}_{12}(t)$ is decreasing and hence is identically zero. In other words, $Q_{12}(x_0,t) \equiv 0$ and so therefore $\alpha_{12}(x_0,t) = 0$ as well. 
\end{proof}

\section{Long time existence}

Let $\underline{\omega}$ be a $\T^3$-invariant hypersymplectic structure on $\T^4$ in symmetric normal form. The main result of this section is that the hypersymplectic flow starting at $\underline{\omega}$ exists for all time (Theorem~\ref{long-time-existence} below). The idea is to show that $\cT$ is uniformly bounded along the flow, and then appeal to the extension result of \cite{FY}, stated as Theorem~\ref{FY-extension} in the introduction. The proof that $\cT$ is bounded takes up the whole section. We begin with a bound on $Q$ itself. 

\begin{proposition}
Let $\underline{\omega}$ be a $\T^3$-invariant hypersymplectic structure on $\T^4$ in symmetric normal form and let $\underline{\omega}(t)$ be the hypersymplectic flow starting at $\underline{\omega}$, which exists on $[0,t_0)$ with $t_0 \leq \infty$. Write $M = \max_{x_0 \in \T^1} \tr Q(x_0)$. Then for any $t \in [0,t_0)$, 
\begin{equation}
\tr Q(x_0, t) \leq M.
\label{trQ-bounded}
\end{equation}
As a consequence, for any $\xi \in \R^3$,
\begin{equation}
\frac{|\xi|^2}{M^2} \leq \xi^T Q\xi \leq M |\xi|^2.
\label{Q-bounded}
\end{equation}
\end{proposition}
\begin{proof}
Taking the trace of~\eqref{dQ/dt} we obtain
\[
(\del_t - \Delta) \tr Q = - \frac{1}{3} \cT \tr Q.
\]
The bound~\eqref{trQ-bounded} now follows from the maximum principle. This implies $Q$ is uniformly bounded above. The two-sided bound~\eqref{Q-bounded} on $Q$ now follows from the fact that $\det Q \equiv 1$ and $Q >0$.
\end{proof}

\begin{proposition}\label{heat-equation-T}
Under the hypersymplectic flow, starting at a $\T^3$-invariant hypersymplecic structure in symmetric normal form, the quantity $\cT$ satisfies the following evolution equation:
\begin{equation}
\begin{split}
(\partial_t -\Delta )\mathcal{T}
	& =  
		- \frac{2}{3}\mathcal{T}^2 
		+
		\frac{8}{V^2}\left(\frac{V'}{V}\right)^2\mathcal{T} 
		-
		\frac{6}{V^4}\frac{V'}{V}\tr\left(
			Q^{-\frac{1}{2}}Q''Q^{-\frac{1}{2}}\cdot Q^{-\frac{1}{2}}Q'Q^{-\frac{1}{2}}
				\right)\\
	& \qquad +
		\frac{2}{V^4} \frac{V'}{V} \tr \left(Q^{-\frac{1}{2}}Q'Q^{-\frac{1}{2}}\right)^3\\
	& \qquad  + 
		\frac{8}{V^4} \tr\left(
			Q^{-\frac{1}{2}}Q''Q^{-\frac{1}{2}}\cdot 
				\left(Q^{-\frac{1}{2}}Q'Q^{-\frac{1}{2}}\right)^2
			\right)\\
	& \qquad + 
		\frac{1}{V^2}\left[
			5\frac{V'}{V}\mathcal{T}'
			- 
			\frac{2}{V^2}\tr \left(Q^{-\frac{1}{2}}Q''Q^{-\frac{1}{2}}\right)^2
			-
			\frac{6}{V^2}
			\tr \left(Q^{-\frac{1}{2}}Q'Q^{-\frac{1}{2}}\right)^4
		\right].
\end{split}
\end{equation}
\end{proposition}
\begin{proof}
We will differentiate directly from $\mathcal{T}= V^{-2}Q_{ik}'Q^{kl}Q_{jl}'Q^{ij}$. Notice that from this formula alone one can see that terms involving $V''$ and $Q'''$ will appear in $\Delta \mathcal{T}$. One of the points of this Propostion is that these terms exactly cancel with counterparts in $\del_t \mathcal{T}$, thanks to the evolution equations of both $Q$ and $V$.

Starting with $\mathcal{T}'$, we have
\begin{align}
\mathcal{T}'
	& = 
		-2V^{-1}V' \mathcal{T}
		+ 
		V^{-2}Q_{ik}'' Q^{kl}Q_{lj}' Q^{ji}
		+ 
		V^{-2} Q_{ik}' Q^{kl}Q_{lj}'' Q^{ji} 
		\nonumber\\
	&\phantom{=}
		\qquad -
		V^{-2}Q_{ik}'Q^{kq}Q_{qp}'Q^{pl}Q_{lj}'Q^{ji}
		- 
		V^{-2}Q_{ik}'Q^{kl}Q_{lj}'Q^{jp}Q_{pq}'Q^{qi}
		\nonumber\\
	& = 
		-2V^{-1}V' \mathcal{T}
		+ 
		2 V^{-2}Q_{ik}'' Q^{kl}Q_{lj}' Q^{ji}
		-
		2V^{-2}Q_{ik}'Q^{kq}Q_{qp}'Q^{pl}Q_{lj}'Q^{ji} 
		\nonumber\\
	& = 
		- 2\frac{V'}{V}\mathcal{T}
		+ 
		\frac{2}{V^2}
		\left[
		\tr \left(
			Q^{-\frac{1}{2}}Q''Q^{-\frac{1}{2}}\cdot Q^{-\frac{1}{2}}Q'Q^{-\frac{1}{2}}
		\right)
		- 
		\tr \left(Q^{-\frac{1}{2}}Q'Q^{-\frac{1}{2}}\right)^3
		\right].
		\label{eqn:T-first-derivative}
\end{align}
Differentiating again we obtain
\begin{align*}
\mathcal{T}''
	&=
		-2\left(V^{-1}V'\right)'\mathcal{T} 
		-
		2V^{-1}V'\mathcal{T}' 
		-
		4 V^{-3}V'\left(Q''_{ik}Q^{kl}Q'_{lj}Q^{ji} - Q'_{ik}Q^{kq}Q'_{qp}Q^{pl}Q'_{lj}Q^{ji}\right) \\
	&\phantom{=} 
		\qquad 
		+ 
		2V^{-2}\Big[
			Q_{ik}'' Q^{kl}Q_{lj}' Q^{ji}
			-
			Q_{ik}'Q^{kq}Q_{qp}'Q^{pl}Q_{lj}'Q^{ji}
			\Big]'\\
	&=
		-2\left(V^{-1}V'\right)'\mathcal{T} 
		-
		2V^{-1}V'\mathcal{T}' 
		-
		2 V^{-1}V'\left(\mathcal{T}'+2V^{-1}V'\mathcal{T}\right) \\
	&\phantom{=} 
		\qquad 
		+ 
		2V^{-2}\Big[
			Q_{ik}'' Q^{kl}Q_{lj}' Q^{ji}
			-
			Q_{ik}'Q^{kq}Q_{qp}'Q^{pl}Q_{lj}'Q^{ji}
			\Big]',
\end{align*}
where we have used that 
\[
2V^{-2}\left(Q''_{ik}Q^{kl}Q'_{lj}Q^{ji} - Q'_{ik}Q^{kq}Q'_{qp}Q^{pl}Q'_{lj}Q^{ji}\right)= \mathcal{T}'+2V^{-1}V'\mathcal{T}. 
\]
Continuing we see that
\begin{align}
\mathcal{T}''
	&= 
		-2\left(V^{-1}V'\right)'\mathcal{T} 
		-
		2V^{-1}V'\mathcal{T}'
		-
		2 V^{-1}V'\left( \mathcal{T}'+2V^{-1}V'\mathcal{T}\right)\nonumber\\
	&\phantom{=} 
		\qquad
		+
		2V^{-2}\Big[
			Q_{ik}''' Q^{kl}Q_{lj}' Q^{ji} 
			- 
			Q_{ik}'' Q^{kq}Q^{pl}Q_{pq}' Q_{lj}' Q^{ji}
			+ 
			Q_{ik}'' Q^{kl}Q_{lj}'' Q^{ji} \nonumber\\
	&\phantom{=} 
			\qquad\qquad 
			- 
			Q_{ik}'' Q^{kl}Q_{lj}' Q^{jq}Q^{pi}Q_{pq}'   
			\Big] \nonumber\\
	&\phantom{=} 
		\qquad\qquad\qquad
			- 2V^{-2} \Big[ 
				Q_{ik}''Q^{kq}Q_{qp}'Q^{pl}Q_{lj}'Q^{ji} 
 				- 
 				Q_{ik}'Q^{kt}Q^{sq}Q_{st}'Q_{qp}'Q^{pl}Q_{lj}'Q^{ji}\nonumber\\
	&\phantom{=}
			\qquad\qquad\qquad\qquad  
				+ 
				 Q_{ik}'Q^{kq}Q_{qp}''Q^{pl}Q_{lj}'Q^{ji}
 				- 
				Q_{ik}'Q^{kq}Q_{qp}'Q^{pt}Q^{sl}Q_{st}'Q_{lj}'Q^{ji}\nonumber\\
	&\phantom{=}
			\qquad\qquad\qquad\qquad 
				+
	 			Q_{ik}'Q^{kq}Q_{qp}'Q^{pl}Q_{lj}''Q^{ji}
				- 
				Q_{ik}'Q^{kq}Q_{qp}'Q^{pl}Q_{lj}'Q^{jt}Q^{si}Q_{st}'
				\Big]\nonumber\\
	&=
		-2\left(\frac{V'}{V}\right)'\mathcal{T} 
		-
		4\frac{V'}{V}\mathcal{T}'
		-
		4 \left( \frac{V'}{V}\right)^2\mathcal{T} 
		+ 
		\frac{2}{V^2}\tr \left(
			Q^{-\frac{1}{2}}Q'''Q^{-\frac{1}{2}}\cdot Q^{-\frac{1}{2}}Q'Q^{-\frac{1}{2}}
				\right) \nonumber\\
	&\phantom{=} 
		\qquad
		- 
		\frac{10}{V^2}\tr \left(
			Q^{-\frac{1}{2}}Q''Q^{-\frac{1}{2}}\cdot \left( Q^{-\frac{1}{2}}Q'Q^{-\frac{1}{2}}\right)^2
		\right)
		+ 
		\frac{2}{V^2} \tr \left(Q^{-\frac{1}{2}}Q''Q^{-\frac{1}{2}}\right)^2\nonumber\\
	&\phantom{=} 
		\qquad\qquad
		+ 
		\frac{6}{V^2}\tr \left(Q^{-\frac{1}{2}}Q'Q^{-\frac{1}{2}}\right)^4. 
		\label{eqn:T-second-derivative}
\end{align}

Next we compute the time derivative of $\mathcal{T}$, again starting from $\mathcal{T} = V^{-2}Q'_{ik}Q^{kl}Q'_{jl}Q^{ij}$:
\begin{align*}
\del_t \mathcal{T}
	&=
		-\frac{2}{3} \mathcal{T}^2 
		+ 
		V^{-2} (\Delta Q_{ik})'Q^{kl}Q'_{jl}Q^{ij} 
		+ 
		V^{-2} Q'_{ik}Q^{kl}(\Delta Q_{jl})'Q^{ij}\\
	&\phantom{=}
		\qquad
			- V^{-2} Q'_{ik}Q^{kq}\Delta Q_{qp}Q^{pl}Q'_{lj}Q^{ji}
			- V^{-2} Q'_{ik}Q^{kl}Q'_{lj}Q^{jp}\Delta Q_{pq}Q^{qi}\\
	&=
		-\frac{2}{3}\mathcal{T}^2 
		+ 
		2V^{-2}(\Delta Q_{ik})' Q^{kl}Q_{jl}'Q^{ij}
		-
		2 V^{-2}Q'_{ik}Q^{kq}\Delta Q_{qp}Q^{pl}Q'_{lj}Q^{ji}.
\end{align*}
Here we have used \eqref{dV/dt} and \eqref{dQ/dt} which give $\del_t V = \frac{1}{3} \mathcal{T}V$ and  $\del_t Q = \Delta Q - \frac{1}{3}  \mathcal{T} Q$. 

We now use  Lemma~\ref{invariant-laplacian-formula} which gives that, for a function $f$ of $x_0$ only, 
\[
\Delta f = V^{-1}(V^{-1}f')' = V^{-2}f'' - V^{-3}V'f'.
\]
This means that
\begin{align}
\del_t \mathcal{T}
	&=
		-\frac{2}{3} \mathcal{T}^2 
		+ 
		2V^{-2}\left( 
		V^{-2}Q''_{ik} - V^{-3}V'Q'_{ik}
		\right)'Q^{kl}Q'_{jl}Q^{ij}\nonumber\\
	&\phantom{=}
		\qquad
			-2 V^{-2} Q'_{ik}Q^{kq}\left( 
			V^{-2}Q''_{qp} - V^{-3}V'Q'_{qp}
			\right)Q^{pl}Q'_{lj}Q^{ji} \nonumber\\
	&=
		-\frac{2}{3} \mathcal{T}^2
		+
		2V^{-2} \left[ 
		V^{-2}Q'''_{ik} - 3V^{-3}V' Q''_{ik} - \left( V^{-3}V'\right)' Q'_{ik}
		\right]Q^{kl}Q'_{jl}Q^{ij}	\nonumber\\
	&\phantom{=}
		\qquad
			-2 V^{-2} Q'_{ik}Q^{kq}\left( 
			V^{-2}Q''_{qp} - V^{-3}V'Q'_{qp}
			\right)Q^{pl}Q'_{lj}Q^{ji} \nonumber\\
	&=	
		-\frac{2}{3} \mathcal{T}^2
		-
		2 \left( V^{-3}V'\right)	' \mathcal{T}
		+
		\frac{2}{V^4}
		\tr \left( Q^{-1/2}Q''' Q^{-1/2} \cdot Q^{-1/2}Q'Q^{-1/2}\right)\nonumber	\\
	&\phantom{=}
		\qquad
			- \frac{6V'}{V^5} \tr \left( Q^{-1/2}Q''Q^{-1/2} \cdot Q^{-1/2}Q'Q^{-1/2}\right)\nonumber\\
	&\phantom{=}
		\qquad\qquad
				-\frac{2}{V^4} \tr \left( 
					Q^{-1/2}Q''Q^{-1/2} \cdot \left( Q^{-1/2}Q'Q^{-1/2}\right)^2
				\right)	\nonumber\\
	&\phantom{=}
		\qquad\qquad\qquad
					+\frac{2V'}{V^5} \tr \left( 
					Q^{-1/2}Q'Q^{-1/2}
					\right)^3.
					\label{eqn:T-time-derivative}
\end{align}
We can now put the pieces together, using equations~\eqref{eqn:T-second-derivative} and~\eqref{eqn:T-time-derivative} for $\mathcal{T}''$ and $\del_t\mathcal{T}$ along with the formula for the Laplacian, $\Delta \mathcal{T} = V^{-2}\mathcal{T}'' - V^{-3}V'\mathcal{T}'$. This gives the stated equation for $(\del_t - \Delta)\mathcal{T}$.
\end{proof}

\begin{theorem}\label{T-bounded}
Suppose the hypersymplectic flow, starting at a $\T^3$-invariant hypersymplectic structure in symmetric normal form, exists for $t \in [0,t_0)$ with $t_0 \leq \infty$. Let $\overline{\mathcal{T}}(t) = \max_{\T^4\times\{t\}} \mathcal{T}$. Then, 
\[
\frac{\diff \overline{\mathcal{T}}}{\diff t} \leq - \frac{1}{3} \overline{\mathcal{T}}^2.
\]
in the sense of barriers. It follows that
\[
\overline{\mathcal{T}}(t) \leq \frac{\overline{\mathcal{T}}(0)}{1 + \frac{1}{3}\overline{\mathcal{T}}(0) t}.
\]
\end{theorem}

\begin{proof}
To ease the notation, we write
\begin{align*}
A &= Q^{-1/2}Q''Q^{-1/2}, \\
B &= Q^{-1/2}Q' Q^{-1/2}.
\end{align*}
Note that $V^2 \mathcal{T} = \tr (B^2)$. 

Now let $s \in [0,t_0)$ and suppose that $\overline{\mathcal{T}}(s) = \mathcal{T}(p,s)$ for some $p \in \T^1$. Using equation~\eqref{eqn:T-first-derivative} and $\mathcal{T}'(p,s)=0$, we see that at $(p,s)$,
\[
\frac{V'}{V}
	=
		\frac{1}{\tr (B^2)}
		\left( 
		\tr \left( AB \right)
		-
		\tr \left( B^3\right)
		\right).
\]
We also have that $\Delta \mathcal{T} \leq 0$ at the point $(p,s)$. It follows from the heat equation of Proposition~\ref{heat-equation-T} that at the spatial maximum $\del_t \mathcal{T}(p,s)$ is bounded above by
\begin{equation}
\begin{split}
-\frac{2}{3}\mathcal{T}^2 
+ 
\frac{2}{V^4 \tr(B^2)}
\Bigg[
	\left( \tr \left(AB\right) \right)^2
	-
	4\tr \left(AB\right) \tr (B^3)
      	+ 
	3 \left( \tr (B^3) \right)^2\\
	+
	4\tr \left(AB^2\right) \tr (B^2)
	-
	\tr (A^2) \tr (B^2)
       	- 
	3  \tr (B^4) \tr (B^2)
\Bigg].
\end{split}
\end{equation}
We will complete the proof by showing that
\begin{equation}
\begin{split}
	\left( \tr \left(AB\right) \right)^2
	-
	4\tr \left(AB\right) \tr (B^3)
      	+ 
	3 \left( \tr (B^3) \right)^2
	+
	4\tr \left(AB^2\right) \tr (B^2) \\
	-
	\tr (A^2) \tr (B^2)
       	- 
	3  \tr (B^4) \tr (B^2)
&\leq 
	\frac{1}{6} \left(\tr(B^2)\right)^3.
\label{bound-AB}
\end{split}
\end{equation}
Assuming this momentarily (and recalling that $\mathcal{T} = V^{-2} \tr(B^2)$) we have that 
\[
\del_t \mathcal{T} (p,s) 
	\leq 
		- \frac{2}{3} \mathcal{T}^2(p,s) + \frac{1}{3} \frac{\left(\tr(B^2)\right)^2}{V^4} (p,s)
	=
		-\frac{1}{3} \mathcal{T}^2(p,s).
\]
For fixed $p$, $\mathcal{T}(p,t)$ is a lower barrier for the Lipschitz function $\overline{\mathcal{T}}(t)$ at $t=s$. Therefore,
\[
\frac{\diff \overline{\mathcal{T}}}{\diff t}(s) \leq \del_t \mathcal{T}(p,s) \leq -\frac{1}{3} \mathcal{T}^2 (p,s) = -\frac{1}{3} \overline{\mathcal{T}}^2(s),
\]
which completes the proof.

To establish the inequality~\eqref{bound-AB}, we use Lemma~\ref{lem:linear-algebra-inequality} below, which proves this inequality for any pair $A,B$ of real symmetric $3\times 3$ matrices with $\tr (B)=0$ and $\tr(A) = \tr (B^2)$. In our situation, differentiating $\det Q=1$ once with respect to $x_0$ shows that $\tr (B)=0$, whilst differentiating it twice shows that $\tr(A) = \tr(B^2)$. 
\end{proof}

\begin{lemma}
\label{lem:linear-algebra-inequality}
Let $A,B$ be two real symmetric $3\times 3$ matrices with $\tr(B)=0$ and $\tr(A)=\tr(B^2)$. Then inequality~\eqref{bound-AB} holds.
\end{lemma}

\begin{proof}
Since the trace is invariant under conjugation, we can assume that $B$ is diagonal. 
Define $\tilde{A}=A-B^2$, then~\eqref{bound-AB} is equivalent to 
\begin{equation}
\left(\tr (\tilde{A} B)\right)^2
- 
2\tr (\tilde{A}B) \tr (B^3) 
+ 
2\tr(\tilde{A}B^2)\tr (B^2)
-
\tr (\tilde{A}^2) \tr( B^2 )
\leq 
\frac{1}{6}\left(\tr (B^2)\right)^3
\end{equation}
provided that $\tr (\tilde{A})=\tr (B)=0$. 

Write $\tilde{A}=\tilde{A}_d+\tilde{A}_o$ where $\tilde{A}_d$ is the diagonal part of $\tilde{A}$ and $\tilde{A}_o$ is the off-diagonal part of $\tilde{A}$; then the above inequality is equivalent to 
\begin{equation}
\begin{split}
\left(\tr (\tilde{A}_d B)\right)^2
 - 
2\tr (\tilde{A}_d B )\tr( B^3 )
+ 
2\tr( \tilde{A}_d B^2)\tr(B^2)\\
\leq 
\left( \tr (\tilde{A}_d^2) + \tr(\tilde{A}_o^2)\right) \tr \left( B^2\right) 
+ 
\frac{1}{6}\left(\tr \left(B^2\right)\right)^3
\end{split}
\end{equation}
provided that  $\tr(\tilde{A}^d)=\tr(B)=0$. By \cite[(3.15)]{HWY}, we actually have that this holds without the need for the $\tr(A_o^2)$ term on the right-hand side:
\begin{equation}
\label{eqn:elementary}
\left(\tr (\tilde{A}_d B)\right)^2
 - 
2\tr (\tilde{A}_d B )\tr( B^3 )
+ 
2\tr( \tilde{A}_d B^2)\tr(B^2)
\leq 
\tr (\tilde{A}_d^2) \tr \left(B^2\right) 
+ 
\frac{1}{6}\left(\tr \left( B^2\right)\right)^3.
\end{equation}
We can also work out this elementary inequality directly. Suppose $\tilde{A}_d= \text{diag}\left(x_1,x_2,x_3\right)$ and $B=\text{diag}\left(y_1,y_2,y_3\right)$, then the right-hand side minus the left-hand side of \eqref{eqn:elementary} can be simplified to 
\begin{align*}
& \left(x_1^2+x_2^2+x_3^2\right)\left(y_1^2+y_2^2+y_3^2\right)
+ 
\frac{1}{6}\left(y_1^2+y_2^2+y_3^2\right)^3\\
&\quad - 
\left(x_1y_1+x_2y_2+x_3y_3\right)^2 
+ 
2\left(x_1y_1+x_2y_2+x_3y_3\right)\left(y_1^3+y_2^3+y_3^3\right)\\
& \qquad - 
2\left(x_1y_1^2+x_2y_2^2+x_3y_3^2\right)\left(y_1^2+y_2^2+y_3^2\right)\\
& = 
3\left[
(x_1y_2-x_2y_1)
- 
\frac{1}{3}(y_1+2y_2)(2y_1+y_2)(y_1-y_2)
\right]^2 
+ 
9 y_1^2y_2^2(y_1+y_2)^2. 
\end{align*}
This last expression is manifestly non-negative and so finishes the proof the lemma.
\end{proof}

\begin{theorem}
Let $\underline{\omega}$ be a $\T^3$-invariant hypersymplectic structure on $\T^4$ in symmetric normal form. Then the hypersymplectic flow starting at $\underline{\omega}$ exists for all time.
\label{long-time-existence}
\end{theorem}
\begin{proof}
By Theorem~\ref{T-bounded}, $\mathcal{T}$ is bounded along the flow and so the flow exists for all time, by the main result of~\cite{FY}.
\end{proof}

\section{Convergence at infinity}

We now show that, modulo diffeomorphisms, the hypersymplectic flow $\underline{\omega}(t)$ converges as $t \to \infty$ to a hyperk\"ahler triple. The overall idea of the argument here is similar to that of~\cite{HWY}. Accordingly we focus more on the parts which are different in our situation.

Recall from Lemma~\ref{metric-when-symmetric} that the corresponding metrics $g_{\underline{\omega}(t)}$ have the form
\[
g_{\underline{\omega}(t)}
	=
		V^2(x_0,t) \diff x_0^2 + Q_{ij}(x_0,t) \diff x_i \diff x_j.
\]
Proving convergence of the metrics amounts to proving convergence of the functions $V$ and $Q_{ij}$. At the same time, it is more straightforward to control geometric quantities, such as curvature and, for a metric of this form, the relation between curvature and $V$ and $Q_{ij}$ is complicated. To get around this, we use a diffeomorphism which puts the metric in a more manageable form.  

For each fixed $t$, we introduce a new coordinate system $(y,x_1,x_2,x_3)$ replacing $x_0$ by $y = y(x_0)$ where
\[
\frac{\diff y}{\diff x_0} = \frac{2\pi}{v_t} V(x_0,t),\qquad y(0)=0,
\]
for $v_t = \int_0^{2\pi} V(x_0,t)\diff x_0$ a $t$-dependent constant. For each $t$ this gives us a diffeomorphism of $\T^1 = \R/2\pi\Z$:
\[
G_t^{-1} \colon \T^1 \to \T^1,\qquad
G_t^{-1}(x_0) = \frac{2\pi}{v_t} \int_{0}^{x_0} V(\xi,t)\diff \xi.
\]
We use the same notation, $G_t^{-1}$, to denote the diffeomorphism of $\T^4$ given by $(x_0,x_1,x_2,x_3)\mapsto (y,x_1,x_2,x_3)$. In the new coordinates, $(y,x_1,x_2,x_3)$, 
\begin{equation}
g_{\underline{\omega}(t)}
	=
		\left( \frac{v_t}{2\pi}\right)^2 \diff y^2 + \widehat{Q}_{ij}(y,t) \diff x_i \diff x_j,
\label{metric-in-new-coordinates}
\end{equation}
where $\widehat{Q}_{ij}(y,t) = Q_{ij}(G_t(y),t)$. 

Equivalently, we use \emph{fixed} coordinates $(y,x_1,x_2,x_3)$ on $\T^4$ and define a new path of hypersymplectic stuctures:
\[
\underline{\widehat{\omega}}(t) = G_t^*\underline{\omega}(t).
\]
Explicitly,
\[
\widehat{\omega}_i(t)
	=	
		\frac{v_t}{2\pi} \widehat{Q}_{ip}(y,t)\, \diff y \wedge \diff x_p + \frac{1}{2}\epsilon_{ijk}\diff x_{jk}.
\]
Now equation~\eqref{metric-in-new-coordinates} says that $g_{\underline{\widehat{\omega}}(t)} = G^*_t g_{\underline{\omega}(t)}$. We will prove the main result of this article, Theorem~\ref{main-result}, by showing that $\underline{\widehat{\omega}}(t)$ converges to a hyperk\"ahler triple as $t \to \infty$. More precisely,

\begin{theorem}\label{convergence}
As $t \to \infty$, the matrix-valued functions $\widehat{Q}_{ij}(y,t) \to \widehat{Q}^\infty_{ij}$ converge in $C^{\infty}$ to a constant positive-definite matrix. Moreover $v_t \to v_\infty$ also converges. It follows that the hypersymplectic structures $\underline{\widehat{\omega}}(t)$ converge in $C^{\infty}$ to the hyperk\"ahler structure
\[
\widehat{\omega}_i^\infty
	=
		\frac{v_\infty}{2\pi} \widehat{Q}^\infty_{ip}\,\diff y \wedge \diff x_p
		+ 
		\frac{1}{2} \epsilon_{ijk} \diff x_{jk}
\]
inducing the flat metric
\[
\widehat{g}^\infty
	=
		\left( \frac{v_\infty}{2\pi} \right)^2 \diff y^2 + \widehat{Q}^\infty_{ij}\, \diff x_i \diff x_j.
\]
\end{theorem} 

Theorem~\ref{main-result} follows immediately. Whilst the diffeomorphisms $G_t$ of Theorem~\ref{convergence} do not start at the identity, replacing them with $G_0^{-1} \circ G_t$ gives a path of diffeomorphisms which does start at the identity and for which the conclusions of Theorem~\ref{main-result} are true. The proof of Theorem~\ref{convergence} takes up the remainder of this section.

Notice that $\underline{\widehat{\omega}}(t)$ is in symmetric normal form for all $t$ with respect to the coordinate system $(y,x_1,x_2,x_3)$, so our previous results for such structures apply directly to $\underline{\widehat{\omega}}(t)$.  

\begin{lemma}
$v_t$ is non-decreasing, bounded above and hence converges to a limit $v_\infty$ as $t\to \infty$.
\end{lemma}

\begin{proof}
By Lemma~\ref{Q-V-formulae} (or directly from~\eqref{metric-in-new-coordinates}, since $\det \widehat{Q} = 1$), the volume form of $g_{\underline{\widehat{\omega}}(t)}$ is
\[
\mu_{\underline{\widehat{\omega}}(t)}
=
\frac{v_t}{2\pi} \diff y \wedge \diff x_{123}.
\]
This implies that 
\[
\Vol(\T^4, g_{\underline{\widehat{\omega}}(t)})
=
(2\pi)^3 v_t.
\]
Meanwhile, since $g_{\underline{\widehat{\omega}}(t)} = G_t^*g_{\underline{\omega}(t)}$ we have that $\Vol(\T^4, g_{\underline{\widehat{\omega}}(t)}) =\Vol(\T^4, g_{\underline{\omega}(t)})$. Now this second quantity is non-decreasing, since $\underline{\omega}(t)$ solves the hypersymplectic flow, and is bounded above, hence it converges.
\end{proof}

\begin{lemma}
The metrics $g_{\widehat{\omega}(t)}$ are uniformly equivalent to the fixed reference metric $g_0 = \diff y^2 + \diff x_1^2 + \diff x_2^2 + \diff x_3^2$.
\end{lemma}

\begin{proof}
We already know that $v_t$ converges as $t \to \infty$ to a strictly positive limit. Moreover, \eqref{Q-bounded} shows that $0<c\leq Q \leq C$ uniformly, form some constants $c, C$. Hence the same is true for $\widehat{Q}= G_t^*Q$. Since these are the metric coefficients of $g_{\underline{\widehat{\omega}}(t)}$ in the coordinate system $y,x_i$ this gives the statement of the Lemma. 
\end{proof}

\begin{lemma}
We have $\widehat{Q}'_{ij} \to 0$ in $C^0$ as $t \to \infty$, where the prime denotes differentiation with respect to $y$.
\end{lemma}

\begin{proof}
The norm-squared $\widehat{\mathcal{T}}$ of the torsion of $\underline{\widehat{\omega}}(t)$ is given by 
\[
\widehat{\mathcal{T}}
=
\left( \frac{2\pi}{v_t}\right)^2 \widehat{Q}'_{ik} \widehat{Q}^{kl} \widehat{Q}'_{lj} \widehat{Q}^{ji},
\]
where the prime denotes differentiation with respect to $y$. (This is the same formula which holds for any hypersymplectic structure in symmetric normal form and ultimately comes from the computation of the torsion forms in the last part of Lemma~\ref{some-coframe-formulae}.) By  Theorem~\ref{T-bounded}, $\mathcal{T} \to 0$ in $C^0$ and so  $\widehat{\mathcal{T}} = G_t^* \mathcal{T}$ also converges to $0$ in $C^0$. Meanwhile~\eqref{Q-bounded} gives constants $c, C$ such that $0< c \leq \widehat{Q} = G_t^*Q \leq C$. Moreover, $v_t \to v_\infty >0$. This implies
\[
\widehat{\mathcal{T}} \geq C \tr\left(\widehat{Q}'^2\right) 
\]
for some constant $C$. Now since $\widehat{\mathcal{T}} \to 0$ in $C^0$ we see that $\tr (\widehat{Q}'^2) \to 0$ in $C^0$ and hence $\widehat{Q}' \to 0$ in $C^0$ as claimed.
\end{proof}

To complete the proof of Theorem~\ref{convergence} we will show that the higher derivatives of $Q_{ij}$ are all bounded. As the following Lemma shows, this amounts to proving uniform $C^{k}$ bounds on the curvature of $g_{\underline{\widehat{\omega}}}$. We write the components of the curvature as
\[
\widehat{R}_{yiy}^{\phantom{yiy}j} 
	= 
		g_{\underline{\widehat{\omega}}}\left( 
			\Rm_{g_{\underline{\widehat{\omega}}}}(\del_y, \del_{x_i})(\del_y), \del_{x_j}
		\right),
\]
and similarly for $\widehat R_{ijk}^{\phantom{ijk}l}$ and $\widehat{R}_{yij}^{\phantom{yij}k}$.

\begin{lemma}
The components of the Riemann curvature tensor of $g_{\underline{\widehat{\omega}}}$ are given by
\begin{eqnarray}
\widehat{R}_{yiy}^{\phantom{yiy}j}
	&=&
		\frac{1}{2} \widehat{Q}^{jk}\widehat{Q}''_{ik}
		-
		\frac{1}{4} \widehat{Q}^{jl}\widehat{Q}'_{lp}\widehat{Q}^{pk}\widehat{Q}'_{ki},
		\label{R-yiyj}\\
\widehat{R}_{ijk}^{\phantom{ijk}l}
	&=&
		\frac{1}{4} \left( \frac{2\pi}{v_t}\right)^2 
		\widehat{Q}^{lp}\left( 
		\widehat{Q}'_{ik}\widehat{Q}'_{jp}
		-
		\widehat{Q}'_{ip}\widehat{Q}'_{jk}
		\right),\\
\widehat{R}_{yij}^{\phantom{yij}k}
	&=&
		0.
\end{eqnarray}
\end{lemma}
\begin{proof}
This is a direct calculation from~\eqref{metric-in-new-coordinates}, using standard formulae for the components of the curvature tensor in local coordiantes. We suppress the details. 
\end{proof}

We make two quick remarks: firstly, we will only use~\eqref{R-yiyj} in what follows, the other components are recorded for completeness; secondly, there is a mistake in equation~(3.3) of~\cite{HWY}, where in an analogous calculation it is erroneously claimed that $R_{jii}^{\phantom{jii}j}=0$. This is ultimately harmless to the overall arguments there.

\begin{lemma}\label{Rm-bounds-Q}
For each $k$ there is a constant $M_k$ such that
\[
\| \widehat{Q}\|_{C^{k+2}(g_0)}
	\leq 
		M_k\left( 
		1  + 
		\left\|
			\Rm(g_{\underline{\widehat{\omega}}(t)})
		\right\|_{C^{k}(g_{\underline{\widehat{\omega}}(t)})}
		\right).		
\]
\end{lemma}

\begin{proof}
Note that the norm on the left-hand side is with respect to the fixed reference metric $g_0 = \diff y^2 + \diff x_1^2 +\diff x_2^2 +\diff x_3^2$ whilst on the right-hand side we use the $t$-dependent norm defined by $g_{\underline{\widehat{\omega}}(t)}$. We first explain how to pass between these norms. 

Recall from~\eqref{metric-in-new-coordinates} that the coefficients of $g_{\underline{\widehat{\omega}}(t)}$ in coordinates $(y,x_1,x_2,x_3)$ are $v_t$ and $\widehat{Q}_{ij}$, and that these are uniformly bounded above and below away from zero. It follows immediately that the $C^{0}(g_0)$ norm on tensors is uniformly equivalent to the $C^{0}(g_{\underline{\widehat{\omega}}(t)})$ norm. Now, given a $C^{k}(g_0)$ bound on $\widehat{Q}$, we obtain a $C^{k-1}$-bound on the Levi-Civita connection matrix of $g_{\underline{\widehat{\omega}}(t)}$ in the coordinates $(y,x_1,x_2,x_3)$. This means that a $C^{k}(g_0)$ bound on $\widehat{Q}$ implies that the $C^{k}$-norms of $g_0$ and $g_{\underline{\widehat{\omega}}(t)}$ are uniformly equivalent (again on tensors).  

We can now prove the Lemma by induction. For the case $k=0$, by~\eqref{R-yiyj}, and the uniform bound $0<c \leq \widehat{Q} \leq C$ we have
\[
|\widehat{Q}''_{ik}| \leq C \left( |R_{yiy}^{\phantom{yiy}j}| + |\widehat{Q}'|^2 \right)
\]
for some constant $C$. Using the fact that $\widehat{Q}$ is uniformly bounded in $C^1(g_0)$ we see that
\[
\|\widehat{Q}\|_{C^2(g_0)} 
	\leq 
		C \left( 
			\|\Rm(g_{\underline{\widehat{\omega}}(t)})\|_{C^0(g_0)}
			+
			1
		\right).
\]
Now the equivalence of the $C^{0}$-norms on tensors means we can replace $C^0(g_0)$ by $C^0(g_{\underline{\widehat{\omega}}(t)})$ on the right-hand side, which proves the inequality for $k=0$. 

Assume inductively we have the bound in the statement of the Lemma for all $k$ up to some $K-1$. We must prove it holds for the $C^{K+2}$-norm of $\widehat{Q}$. The bound for $k=K-1$ implies that the $C^{K+1}$-norms of $g_0$ and $g_{\underline{\widehat{\omega}}(t)}$ are uniformly equivalent with constants of equivalence that depend only on a $C^{K-1}$ bound on $\Rm(g_{\underline{\widehat{\omega}}(t)})$. Now from~\eqref{R-yiyj} we see that a $C^{K}(g_0)$ bound on $\Rm(g_{\underline{\widehat{\omega}}(t)})$ and a $C^{K+1}(g_0)$ bound on $\widehat{Q}$ will imply the $C^{K+2}(g_0)$ bound on $\widehat{Q}$ we seek. By induction, a $C^{K}(g_{\underline{\widehat{\omega}}(t)})$ bound on $\Rm(g_{\underline{\widehat{\omega}}(t)})$ implies a $C^{K+1}(g_0)$ bound on $\widehat{Q}$ and, by equivalence of norms, it also implies the $C^K(g_0)$-norm of $\Rm(g_{\underline{\widehat{\omega}}(t)})$ is bounded. Hence we obtain the required $C^{K+2}(g_0)$ bound on $\widehat{Q}$. 
\end{proof}

With this Lemma in hand, we have reduced the proof of Theorem~\ref{convergence} to showing uniform bounds on the curvature of $g_{\underline{\widehat{\omega}}(t)}$. To do this it turns out to be more efficient to pass to the 7-dimensional manifold $\T^4 \times \T^3$. Recall from \S\ref{introduction} that this manifold carries a path of $G_2$-structures:
\[
\widehat{\phi}(t) 
	= 
		\diff t^{123} 
			- 
			\diff t^1 \wedge \widehat{\omega}_1
			- 
			\diff t^2 \wedge \widehat{\omega}_2
			- 
			\diff t^3 \wedge \widehat{\omega}_3.
\]
We wil use bold symbols to denote tensors defined on $\T^4 \times \T^3$, so that $\mathbf{g}(t)$ denotes the metric associated to $\widehat{\phi}(t)$, $\mathbf{Rm}$ denotes its curvature tensor and so forth. (This is in keeping with the notational conventions of~\cite{FY}.)

\begin{proposition}\label{Lambda-bounded}
There is a constant $C$ such that for all $t$,
\[
\| \mathbf{Rm}\|_{C^0} + \| \mathbf{\nabla} \mathbf T\|_{C^0}
\leq 
C,
\]
where the norms are taken with respect to $\mathbf{g}(t)$.
\end{proposition}
\begin{proof}
We prove this by contradiction and a rescaling argument, exactly as in~\cite{FY}. We sketch the argument here, focusing on the two parts which are different in our situation: the non-collapsing and the final way a contradiction is obtained.

To prepare the ground, we prove a lower bound on volume of balls. Given $p \in \T^4$, let 
\[
\Vol(p,r;t) = \Vol (B(p,r), g_{\underline{\widehat{\omega}}(t)})
\]
denote the volume of the geodesic ball of radius $r$ centered at $p$ determined by the metric $g_{\underline{\widehat{\omega}}(t)}$. We first claim that there exists $r_0>0$ and $c>0$ such that for all $t$ and $r \in [0,r_0]$, 
\[
\Vol(p,r;t) \geq c r^4.
\]
To see this note that cube-shaped domain $\Omega_r = \{(y,x_1,x_2,x_3) : |y| \leq r, |x_i| \leq r\}$ is mapped injectively to $\T^4$ as long as $r \leq \pi$. Moreover, the length of the curve $s \mapsto (\bar{y} + as, \bar{x}_i+ \xi_i s)$ is
\[
\int_0^r  \sqrt{ a^2 + \xi^T \widehat{Q}(as,t) \xi} \diff s \leq Cr
\]
for some $r$, as long as $|a| \leq 1$ and $|\xi| \leq 1$. So the image of $\Omega_r$ is contained in the geodesic ball of radius $Cr$. This implies 
\[
\Vol(p,r;t) \geq \int_{\Omega_{r/C}} \frac{v_t}{2\pi} \diff y \wedge \diff x_{123} \geq c r^4
\]
for some $c$, as long as $r \leq \pi/C$.

This lower bound on volume, together with the Cheeger--Gromov--Taylor inequality~\cite{CGT} implies that, for some constant $i_0$, we have a uniform lower bound on the injectivity radius of the form
\begin{equation}
\inj (\T^4, g_{\underline{\widehat{\omega}}(t)})
	\geq
		\frac{i_0}{\left\| \Rm(g_{\underline{\widehat{\omega}}(t)})\right\|^{1/2}_{C^0} + 1}.
\label{injectivity-bound}
\end{equation}

We now explain how---assuming the statement of the proposition is false--- this enables us to take a rescaled limit. To ease notation, write
\[
\Lambda(q,t) =\left(  |\mathbf{Rm}| + |\mathbf{\nabla}\mathbf{T}|\right)(q,t).
\]
Assuming the result is false, there exist a sequence of points $q_k \in \T^4\times \T^3$ and times $t_k \to \infty$ such that 
\[
\Lambda(q_k,t_k) := \sup_{t \in [0, t_k]} \Lambda(q,t) \to \infty.
\]
We abbreviate $\Lambda_k = \Lambda(q_k,t_k)$. We now define a rescaled sequence of flows by
\begin{align*}
\phi^{(k)}(t) 
	&= 
		\Lambda^{3/2}_k \widehat{\phi}( \Lambda^{-1}_k t + t_k),\\
\underline{\omega}^{(k)}(t)
	&=
		\Lambda_k \underline{\widehat{\omega}}(\Lambda^{-1}_k t + t_k).
\end{align*}
(We remark that there is a mistake in both \cite{FY,FY2} and~\cite{HWY}, where in analogous discussions the scaling of $\phi^{(k)}$ is with the factor $\Lambda_k$ rather than $\Lambda_k^{3/2}$. In the proof of Theorem 4.10 of \cite{FY2} there is also a missing factor $K^2$ in front of $t$ in  $\omega'(t)$ ). The $\phi^{(k)}$ are a sequence of flows defined for $t \in [-\Lambda_k t_k, \infty)$ which are diffeomorphic to $G_2$-Laplacian flows. The rescaling ensures that the corresponding metrics in 7 and 4 dimensions scale as $g_{\phi^{(k)}}  = \Lambda_k g_{\phi}$ and $g_{\underline{\omega}^{(k)}} = \Lambda_k g_{\underline{\omega}}$ respectively. In particular,
\[
\| \mathbf{Rm}(\phi^{(k)}(t))\|_{C^0} 
+ 
\| \mathbf{\nabla} \mathbf{T}(\phi^{(k)}(t))\|_{C^0} 
\leq
1
\]
on $t \in[-\Lambda_k t_k,0]$. Now the Shi-type estimates of Lotay--Wei~\cite{LW1} for the $G_2$-Laplacian flow give that for any $A >0$ and $l$, we have 
\[
\|\mathbf{\nabla}^l \mathbf{Rm}(\phi^{(k)}(t))\|_{C^{0}}
+
\| \mathbf{\nabla}^{l+1} \mathbf{T}(\phi^{(k)}(t))\|_{C^{0}} \leq C_{A,l}
\]
for some constant $C_{A,l}$. (It is the fact that these Shi-type estimates are known for the $G_2$-Laplacian flow, as opposed to the hypersymplectic flow, which forces us to pass somewhat artificially to the 7-manifold.)

From here, arguing as in~\cite{FY}, we obtain uniform bounds on the following \emph{four}-dimensional quantities, again over $t \in [-A,0]$:
\[
\| \Rm(\underline{\omega}^{(k)}(t))\|_{C^l}, \| \underline{\omega}^{(k)}(t)\|_{C^l} \leq C_{A,l}.
\]
Moreover, after rescaling the injectivity radius bound~\eqref{injectivity-bound} gives a uniform lower bound $\inj(\T^4, \underline{\omega}^{(k)}(t)) \geq c$. So we can now take a pointed Cheeger--Gromov limit
\[
(\T^4, \underline{\omega}^{(k)}(0), p_k) \to (X_\infty, \underline{\omega}_\infty, p_\infty)
\] 
where $(X_\infty, \underline{\omega}_{\infty})$ is hyperk\"ahler. The points $p_k \in \T^4$ here are the $\T^4$ projections of the initial points $q_k \in \T^4 \times \T^3$ where the initial blow-up was assumed to happen. (The details of this argument can be found in the final section of~\cite{FY}.)

Our choice of points and rescaling ensures that $\Rm(g_{\underline{\omega}_{\infty}})(p_\infty) \neq 0$ and it is this that we will now show leads to a contradiction. Directly from the expression~\eqref{metric-in-new-coordinates} for  $g_{\underline{\widehat{\omega}}(t)}$ we see that for any $(x_1,x_2,x_3)$ the curve $s \mapsto (s,x_1,x_2,x_3)$ is a geodesic. Moreover it has length uniformly bounded below away from zero. This means that for the rescaled metrics determined by $ \underline{\omega}^{(k)}(t)$, each point $p_k$ lies on a geodesic whose length tends to infinity. So in the limit, $X_{\infty}$  must contain a geodesic line. Since it is Ricci-flat, Cheeger--Gromoll's splitting theorem~\cite{CG} tells us that it is isometric to $Y^3 \times \R$. But now, since it is hyperk\"ahler, this forces it to be flat, giving us our contradiction.
\end{proof}

We can now finish the proof of convergence of the hypersymplectic flow. 

\begin{proof}[Proof of Theorem~\ref{convergence}]
From Proposition~\ref{Lambda-bounded} together with Lotay--Wei's Shi-type estimates~\cite{LW1} we see that there is a constant $C_k$ such that for all $t \in [0,\infty)$,
\[
\| \mathbf{\nabla}^k \mathbf{Rm}(\widehat{\phi}(t))\|_{C^0}
+
\| \mathbf{\nabla}^{k+1} \mathbf{T}(\widehat{\phi}(t))\|_{C^0}
	\leq 	
		C_k.
\]
Corollary~3.3 of~\cite{FY} now implies that
\[
\| \nabla^k \Rm(g_{\underline{\widehat{\omega}}(t)})\|_{C^0} \leq C_k.
\]
Lemma~\ref{Rm-bounds-Q}  implies that $\widehat{Q}$ is bounded in $C^k$ for all $k$. Arzela--Ascoli now implies that for any sequence of times $t_k \to \infty$, there is a subsequence $t_{k_n}$ for which $\widehat{Q}({t_{k_n}} )\to \widehat{Q}^\infty$ converges in $C^{\infty}$. Since $\widehat{Q}' \to 0$, the limit $\widehat{Q}^{\infty}$ is constant. Moreover, the value of $\widehat{Q}^{\infty}$ is independent of the sequence we started with. This is because the following integral is cohomological and so independent of time:
\[
\int_{\T^4} \omega_i(t) \wedge \omega_j(t)
	=
		\int_{\T^4}\widehat{\omega}_i(t) \wedge \widehat{\omega}_{j}(t)
			=
				\int_{\T^4} 2 \widehat{Q}_{ij}(y,t) \frac{v_t}{2\pi} \diff y \diff x_{123}.
\]
From this we deduce that
\[
\widehat{Q}_{ij}^\infty = \frac{2\pi}{v_\infty} \int_0^{2\pi} \alpha_{ij}(x_0,0) \diff x_0
\]
is determined purely by the starting data $\underline{\omega}(0)$. It now follows that the whole path converges: $\widehat{Q}(t) \to \widehat{Q}^{\infty}$ in $C^\infty$.  For if not, there would exist $\epsilon>0$, a sequence of times $t_k \to \infty$ and an integer $l$ for which 
\[
\| \widehat{Q}(t_k) - \widehat{Q}^\infty\|_{C^l} \geq \epsilon.
\]
Then no subsequence of $\widehat{Q}(t_k)$ could converge to $\widehat{Q}^{\infty}$, a contradiction which completes the proof. 
\end{proof}

	\end{document}